\documentclass[a4paper,secnum,leqno]{article}

\usepackage{amssymb, amsmath, amsthm}
\usepackage{graphics}
\usepackage[dvips]{graphicx}
\usepackage{hyperref}
\usepackage[margin=2.5cm]{geometry}

\newtheorem{thm}{Theorem}[section]

\newtheorem{prp}[thm]{Proposition}

\theoremstyle{definition}
\newtheorem{dfn}[thm]{Definition}
\newtheorem{exa}[thm]{Example}

\theoremstyle{remark}
\newtheorem{rem}{Remark}

\def\dist{\operatorname{d}}
\def\epsilon{\varepsilon}
\def\GL{\operatorname{GL}}
\def\Card{\operatorname{Card}}

\def\AAA{\mathcal{A}}
\def\BBB{\mathcal{B}}

\def\NN{\mathbb{N}}
\def\ZZ{\mathbb{Z}}
\def\RR{\mathbb{R}}
\def\TT{\mathbb{T}}

\title{$S$-adic expansions: a combinatorial, arithmetic and geometric approach}
\author{Val\'erie \textsc{Berth\'e\footnote{LIAFA, Universit\'e Paris Diderot
Paris 7 - Case 7014
F-75205 Paris Cedex 13, France.\newline e-mail: \texttt{berthe@liafa.univ-paris-diderot.fr}}\,}
 and Vincent \textsc{Delecroix}\footnote{IMJ, Universit\'e Paris Diderot
Paris 7 - Case 7012
F-75205 Paris Cedex 13, France.\newline e-mail: \texttt{delecroix@math.univ-paris-diderot.fr}}}
 
\begin{document}
\title{Beyond substitutive dynamical systems: $S$-adic expansions}
\maketitle

\begin{abstract} 
An $S$-adic expansion of an infinite word is a way of writing it as the limit of an infinite product of substitutions (i.e., morphisms of a free monoid). Such a description is related to continued fraction expansions of numbers and vectors. Indeed, with a word is naturally associated, whenever it exists, the vector of frequencies of its letters. A fundamental example of this relation is between Sturmian sequences and regular continued fractions. In this survey paper, we study $S$-adic words from different perspectives, namely word combinatorics, ergodic theory, and Diophantine approximation, 
 by stressing the parallel with continued fraction expansions.
\end{abstract} 

\tableofcontents

\section{Introduction}
Highly ordered structures may be obtained by the following procedure: start from a pattern of basic elements (this can be tiles or words, for instance), and repeatedly replace each basic element by a union of them according to a fixed rule. Consider for example the Fibonacci substitution $\sigma$ defined on the set of finite words on $\{a,b\}$ which replaces $a$ with $ab$ and $b$ with $a$. Starting from $a$ and applying repeatedly $\sigma$ we obtain the sequence of finite words $a$, $ab$, $aba$, $abaab$, \ldots that converges to the infinite word $abaababaab\ldots$ called the Fibonacci word. The set or words produced by substitutions are self-similar, i.e., they can be decomposed as a finite disjoint union of their images under some endomorphisms. Note that the notion of substitution occurs in various fields of mathematics (symbolic dynamics, ergodic theory, number theory, harmonic analysis, fractal theory), as well as in theoretical computer science (combinatorics on words and language theory), physics (quasicrystalline structures), and even in biology (bioinformatics, plant development), such as illustrated by the bibliography of~\cite{pyt}.

If one wants to go beyond the substitutive case, and thus, get more flexibility in the hierarchical structure, one might want to change the substitution rules at each iteration step. We get systems which have similarities which depend on the scale. This leads to the notion of \emph{$S$-adic representation} of an infinite word $w$. A \emph{purely substitutive word} is a limit $w = \displaystyle \lim_{n \to+ \infty} \sigma^n(a),$ where $\sigma: \AAA^* \rightarrow \AAA^*$ is a substitution and the letter $a \in \AAA$. An example of a purely substitutive word is the Fibonacci word. An $S$-adic expansion of an infinite word $w$ is a sequence $((\sigma_n,a_n))_{n \in \NN}$, where $\sigma_n: \AAA_{n+1}^* \rightarrow \AAA_n^*$ are substitutions and the letters $a_n$ belong to $ \AAA_n$, such that $$w = \lim_{n\to+\infty} \sigma_0 \, \sigma_1 \, \cdots \, \sigma_{n-1}(a_n).$$ Hence, a purely substitutive word is an $S$-adic word with a periodic $S$-adic expansion. Further natural requirements have to be added to the sequence $((\sigma_n,a_n))_{n \in \NN}$ in order that the limit makes sense and to forbid degenerate constructions (see Remark~\ref{rem:cass}, p.~\pageref{rem:cass}). We emphasize that the notion of $S$-adic expansion is dedicated to words with low factor complexity (that is, with a low number of factors of a given length); indeed, under some general conditions, an $S$-adic word has zero entropy (see Theorem~\ref{thm:zero_entropy}).

A first parallel with continued fraction expansions can already be drawn: infinite words are analogues of real numbers and substitutions play the role of digits or partial quotients. Moreover, an $S$-adic expansion of a word gives naturally a generalized continued fraction expansion of its letter frequencies: the vector of uniform letter frequencies of the word (if any) is the limit of an infinite product of non-negative integer matrices obtained by taking an abelianized version of the substitutions, i.e., their incidence matrices (see Theorem~\ref{thm:cone_and_frequencies} and Theorem \ref{thm:invariant_measures}).

The term `adic' refers to the inverse limit for the products (with respect to composition) of substitutions of the form $\sigma_0 \, \sigma_1 \, \cdots \, \sigma_{n-1}$, and the terminology $S$-adic, introduced by S.~Ferenczi~\cite{frank}, comes from A.~Vershik's terminology for adic systems~\cite{Ver82}, with the letter $S$ referring to `substitution'.
Roughly speaking, an adic transformation acts on a path space made of infinite paths in a graph (a Bratteli-Vershik diagram,
or a Markov compactum), by associating with a path its successor with respect to a lexicographic ordering.
Recall that any invertible dynamical system on a Cantor space is topologically conjugate to the successor
map of a properly ordered Bratteli diagram~\cite{HPS}.
This description through Bratteli diagrams leads to the characterization of the topological orbit
equivalence~\cite{GPS}, or to the description of the stationary ordered Bratteli diagrams as
the (disjoint) union of the set of substitutive minimal dynamical systems and the set of
stationary odometers~\cite{DHS99}. An equivalent way of viewing Bratteli diagrams is through
Kakutani-Rokhlin constructions. The $S$-adic formalism can be viewed as a combinatorial
counterpart of these constructions where dynamical systems are studied through symbolic codings.

Families of $S$-adic words appear in different ways. The most simple one consists in taking a finite set of substitutions and consider all limits we can obtain by composing these substitutions. A more involved version that we develop uses graph-directed constructions. Examples of $S$-adic families include Sturmian systems~\cite[Chap. 2]{Lothaire:2002}, \cite[Chap. 5]{pyt}, Arnoux-Rauzy words~\cite{AR91,CFM08,CFZ00}, interval exchange transformations~\cite{ra1}. Moreover some general classes of dynamical systems 
admit  $S$-adic expansions with `nice' properties such as linearly recurrent systems (characterized in~\cite{Dur03} as the primitive proper $S$-adic subshifts), and, more generally, systems of finite rank in the measure-theoretic sense (see e.g.~\cite{frank} and the survey~\cite{frank2}).

\subsubsection*{History and motivation.}
Working with $S$-adic expansions has occurred quite early in the literature under various forms. (For more on historical considerations, see e.g. \cite[Chap. 12]{pyt} and \cite{Dur03,Leroy}.) Let us mention for instance~\cite{Wen} for particular $S$-adic expansions generated by two substitutions over a two-letter alphabet, or else~\cite{Mozes}, where consistent rectangular tiling substitutions are composed.

Note that the need to go beyond the self-similar case also occurs in various fields. One finds for instance constructions that go beyond Hutchinson's theory for iterated function systems in fractal theory: different contraction transformations can be applied at each step, with the contraction ratios being possibly different at each level; see e.g. \cite{MU} for the notion of conformal infinite iterated function systems, and also \cite{Hua}, for the notion of Moran sets and class, and for a review of various related notions. The same idea has also been developed in the context of numeration systems. Indeed, fibred numeration systems consist in consecutive iterations of a transformation producing sequences of digits. A simple generalization is obtained by changing the transformation at any step, with, as a classic example, Cantor (that is, mixed radix) expansions. For more details, see the survey \cite{BBLT}, and more generally, for the notion of fibred systems, see \cite{Schweiger95}.

The notion of $S$-adic expansion together with the terminology has been made precise by S.~Ferenczi in~\cite{frank}. The aim of this paper was to investigate the relation between the factor complexity of an infinite word (that is, the number of its factors of a given length), and the measure-theoretic rank of the associated dynamical system (that is, the number of Rokhlin towers required to approximate it). The notion of (finite) rank in ergodic theory has then lead to various adic-type symbolic constructions; see for instance \cite[Chap. 4,6,7]{CANT} and the references therein, or else \cite{FFT}, in the flavor of standard rules in word combinatorics \cite[Chap. 2]{Lothaire:2002}. Note also the more recent notion of fusion tilings introduced in \cite{PS11}, which also includes Bratteli-Vershik systems as well as multidimensional cutting and stacking  transformations. 

Expansions of $S$-adic nature have now proved their efficiency for yielding convenient descriptions for highly structured symbolic dynamical systems (opposed to chaotic or unpredictable ones). If one wants to understand such a system (for instance, its invariant measures), it might prove to be convenient to decompose it via a desubstitution process: an $S$-adic system is a system that can be indefinitely desubstituted. This approach has proved its strength for symbolic systems with linear factor complexity (see Section \ref{sec:complexity} below for more details). The desubstitution process can often be seen in terms of inductions, that is, first return maps. This is the seminal approach undertaken by G.~Rauzy in~\cite{ra0,ra1} for interval exchanges, and as a further example, fruitfully developed in~\cite{CFM08,CFZ00} for the study of Arnoux-Rauzy words.

Note also that numerous properties might then be deduced from a `right' $S$-adic representation of a dynamical system; consider as an illustration a circle rotation with its combinatorial counterpart, namely a Sturmian system (see Example~\ref{exa:sturm}), which yields to the study of occurrences of powers of factors~\cite{ice}, discrepancy estimates~\cite{Adam2,Adam3}, or else, the construction of transcendental numbers through their expansions (see e.g. the survey \cite[Chap. 8]{CANT} and the references therein).

\subsubsection*{Contents.} We have chosen to focus in the present survey paper on the parallel between continued fractions and $S$-adic words. This parallel originates in the seminal study of Sturmian dynamical systems, which are completely described in terms of $S$-adic expansions governed by the continued fraction expansion. For more details, see \cite{afh,AF1,ice,LRi}, \cite[Chap. 2]{Lothaire:2002}, or else, \cite[Chap. 5]{pyt}.

Let us briefly sketch the contents of the present paper. Basic definitions and preliminaries in symbolic dynamics will be recalled in Section \ref{sec:def}, by focusing on substitutive dynamics (that corresponds to periodic $S$-adic expansions) and shifts of finite type (that will be used to describe families of $S$-adic systems). Section~\ref{sec:representation} is devoted to the introduction of the definitions related to $S$-adic expansions. We consider in Section~\ref{sec:graph} (and later also in Section~\ref{subsec:definition_Lyapunov})
graph-directed $S$-adic systems, that is, families of $S$-adic words generated by $S$-adic 
graphs, with, as a main motivation, the fact that dynamical quantities associated with a graph-directed $S$-adic system (such as the entropy and Lyapunov exponents) are linked to generic properties of the $S$-adic words that belong to this system. The parallel with continued fractions expansions is stressed in Section~\ref{sec:cf}. We then handle the notion of $S$-adic expansions from a combinatorial, ergodic and arithmetic prospective. The combinatorial viewpoint is first tackled in Section \ref{sec:complexity} with a special focus on the factor complexity. Dynamical properties such as primitivity, linear recurrence, frequencies and invariant measures are discussed in Section~\ref{sec:dynamics}. Lastly,
 a quantitative and arithmetic viewpoint, including strong convergence issues and balance properties of words, together with deviations of Birkhoff averages, is handled in Section~\ref{sec:Lyapunov_exponents}.

\subsubsection*{What this paper is not about.}
In this paper we focus on certain topics while letting others aside. We try now to provide references to other related works on $S$-adic systems.

The theory of interval exchanges and the associated translation surfaces is a research  field which has been  very active during  the past 30 years. If interval exchanges and their symbolic codings may be thought as particular cases of $S$-adic systems,  there is not enough room here  to speak about the tremendous list of results that are specific to them. Let us just  mention the most striking ones: the existence of nice invariant measures on the space of interval exchanges~\cite{Masur,Veech},  which leads to the resolution of Keane's conjecture on the genericity of unique ergodicity, the prevalence of weak-mixing~\cite{Avila}, continued fractions and associated $S$-adic expansions of trajectories of billiards in regular polygons~\cite{ArnouxHubert,SmillieUlcigrai,Ferenczi-2ngon},  and lastly, codings of rotations and modular flow~\cite{Arnoux-modulaire,AF1}.

\smallskip

We will not consider the spectral theory of dynamical systems,  and in particular, the existence of (topological or measurable) eigenvalues and the question of weak-mixing.
For linearly recurrent systems, topological and measurable eigenvalues are completely characterized~\cite{Cortez,Bressaud,Bressaud2}. Note that  the formulation of this criterion is $S$-adic in nature.

The Pisot conjecture is a conjecture about the characterization of purely discrete spectrum for purely substitutive systems~\cite{pyt,CANT} (or $\beta$-numerations) of Pisot type. It is related to the notion of Pisot number through the eigenvalues of the incidence matrix of the substitution.  In particular cases, namely substitutive Arnoux-Rauzy words, the
 Pisot conjecture is known to hold~\cite{arnito,BJS}. Even for some $S$-adic Arnoux-Rauzy systems the purely discrete spectrum is still valid~\cite{BST}. Nevertheless, it is not true for all Arnoux-Rauzy words as proved in~\cite{CFM08}: there are examples of Arnoux-Rauzy systems that are weak-mixing.  The program initiated by G.Rauzy to prove that conjecture uses the so-called Rauzy fractal.
  The theory of Rauzy fractals has now proved its relevance in this framework and they  have many applications in various fields that range from number theory to discrete geometry; see for instance \cite{SieThus} or \cite[Chap. 5]{CANT} for a survey of the literature.
 However the parameters associated with Rauzy fractals for Pisot substitutions  are algebraic numbers. The question is now to be able to reach non-algebraic parameters, by working with suitable multidimensional continued fraction algorithms and their associated $S$-adic expansions; see \cite{BST} for the existence of Rauzy fractals associated with almost every two-dimensional toral translation, and see also Section \ref{sec:cf} for examples of $S$-adic expansions associated with multidimensional continued fractions.

\smallskip

Lastly, we will   not handle  multidimensional words. The notion of $S$-adic expansion may be indeed extended in that setting; see for example \cite{Aubrun} for a characterization of sofic $S$-adic multidimensional subshifts in terms of effectiveness of the directive sequence of substitutions.

\section{Preliminaries in symbolic dynamics}\label{sec:def}
\subsection{Words, languages and shifts}
Let $\AAA = \{1,2,\ldots,d\}$ be a finite alphabet.
Let $\AAA^*$ be the set of finite words, i.e., $\AAA^0 \cup \AAA^1 \cup \ldots$, and $\AAA^\NN$
the set of infinite words. The set $\AAA^\NN$ shall be equipped
with the product topology of the discrete topology on each copy of $\AAA$; this topology is
induced by the following distance: for two distinct infinite words $u$ and $v$ in $\AAA^\NN$, 
$\dist(u,v)=2^{-\min\{n \in \NN \ \mid \ u_n\neq v_n\}}$. Thus, two sequences are close to each other if
their first terms coincide. This topology extends to the set $\AAA^* \cup \AAA^\NN$ 
of finite and infinite words, with $\AAA^*$ being dense in $\AAA^* \cup \AAA^\NN$. It makes precise
the convergence of finite words to infinite words.

The set~$\AAA^*$ endowed with the operation of concatenation
forms a monoid. For a word~$u$ in $\AAA^*$, $|u|$~denotes its length and~$|u|_i$ the
number of occurrences of the letter~$i$ in~$u$. A~\emph{factor} of a finite or infinite word~$u$
is defined as the concatenation of consecutive letters occurring in~$u$. The set of factors of $u$ is called
its \emph{language}, and is denoted as $L_u$.
 
Let $T$ stand for the (one-sided) \emph{shift} on $\AAA^\NN$,
i.e., $T((u_n)_{n \in \NN}) = (u_{n+1})_{n \in \NN}$.
The shift $T$ is continuous and surjective on $\AAA^\NN$.
A \emph{symbolic dynamical system}, or \emph{shift}, or else \emph{shift space},  is a non-empty closed
subset of~$\AAA^\NN$, stable under the action of the shift. By considering factors of elements of a given shift
$X$, we may define its \emph{language}~$L(X) \subset \AAA^*$. The set~$L(X)$ is \emph{factorial} (it contains
all the factors of any of its elements), and \emph{extendable} (for any element $u$ of $L(X)$, there
exists a letter~$a$ such that~$wa$ is also in~$L(X)$). There is a well-known equivalence between shifts on
$\AAA$ and factorial extendable languages with letters in $\AAA$ (see e.g.~\cite{Lind}
or~\cite[Proposition 7.1.2]{CANT}), and we pass freely from one to the other: if $L$ is a factorial extendable
set of words, then the \emph{shift} $X_L$ defined by $L$ is the set of infinite words in $\AAA^{\NN}$ whose
factors belong to $L$. 

Let $u$ be an infinite word in~$\AAA^\NN$. Let $X_u$ be the orbit closure of the infinite word $u$ under the
action of the shift~$T$. More precisely, $X_u$ is  the closure in $\AAA^\NN$ of the set
$\{T^n(u) \mid n \in \NN\} = \{(u_k)_{k \geq n} \mid k \in \NN\}$.
The set~$X_u$ is closed and $T$-invariant and hence a shift. It coincides with the  set of infinite words whose
language is contained in $L_u$. We refer to $X_u$ as the \emph{symbolic dynamical system generated by~$u$}.

We will also sometimes consider biinfinite words, i.e.,  elements of $\AAA^\ZZ$.

\subsection{Substitutive systems and shifts of finite type}\label{subsec:subs_and_finite_type}
In this section we describe two types of, somewhat opposite, symbolic dynamical systems, namely self-similar
dynamical systems determined by a substitution and, on the other hand, shifts of finite type described
by finitely many adjacency rules. Their properties are very different, but in the context of substitutive
dynamics, and more generally of $S$-adic dynamics, there is a constant interplay between these two types of dynamics.

A morphism of the free monoid $\sigma \colon \AAA^* \rightarrow \BBB^*$ is completely determined by the images
of the letters in $\AAA$. We will often call such  a morphism a \emph{substitution}. A substitution
$\sigma: \AAA^* \rightarrow \BBB^*$ is \emph{non-erasing} if there is no letter in $\AAA$ whose image
under $\sigma$ is the empty word. We assume here that all substitutions are non-erasing. 
 
Let $\sigma: \AAA^* \rightarrow \BBB^*$ be a substitution where $\AAA = \{1,\ldots,d\}$ and $\BBB = \{1,\ldots,d'\}$.
The \emph{incidence matrix} of $\sigma$ is the $d' \times d$ matrix $M_\sigma = (m_{i,j})_{i,j}$ with entries
$m_{i,j}:=| \sigma(j)|_i$. It is a non-negative integer matrix.
A substitution $\sigma$ over the alphabet ${\cal A}$ is said to be \emph{primitive} if there exists a positive
integer $k$ such that, for every $i$ and $j$ in $\AAA$, the letter $i$ occurs in $\sigma^k(j)$.
It is \emph{positive} if, for every $i$ and $j$ in $\AAA$, the letter $i$ occurs in $\sigma(j)$, or equivalently,
all entries of the incidence matrix $M_{\sigma}$ are positive.

Substitutions are very efficient tools for generating infinite words and describing dynamical systems. Indeed,
let $\sigma$ be a non-erasing substitution over the alphabet $\AAA$. Assume that there exists a letter $i$ such
that $\sigma(i)$ begins with $i$ and $|\sigma(i)| \geq 2$. It is easily seen that
$\lim_{n\to+\infty}|\sigma^n(i)|=+\infty$. Then there exists a unique fixed point $u$ of $\sigma$ beginning with $i$.
This infinite word is obtained as the limit in $\AAA^* \cup \AAA^\NN$ of the sequence of words $(\sigma^n(a))_n$.
Such an infinite word is said to be \emph{generated by} $\sigma$, and more generally, it is said to be
\emph{purely substitutive} or \emph{purely morphic}. If $u \in \AAA^\NN$ is purely morphic and if
$\tau: \AAA^* \to \BBB^*$ is a morphism, then the word $v=\tau(u)$ is said to be \emph{morphic} or
\emph{substitutive}.

\begin{exa}[Fibonacci substitution] \label{ex:fibonacci}
We consider the substitution $\sigma$ on $\AAA = \{a,b\}$ defined by $\sigma(a)=ab$ and $\sigma(b)=a$.
Then, the sequence of finite words $(\sigma^n(a))_n$ starts with
\[
\sigma^0(a) = a, \ \sigma^1(a) = ab, \ \sigma^2(a) = aba, \ \sigma^3(a) = abaababa, \ \ldots
\]
Each $\sigma^n(a)$ is a prefix of $\sigma^{n+1}(a)$, and the limit word in $\AAA^\NN$ is
\[
abaababaabaababaababaabaababaabaababaababaabaababaababaabaab\ldots
\]
The above limit word is called the \emph{Fibonacci word} (see e.g. \cite{Lothaire:2002,pyt}).
\end{exa} 

If $\sigma$ is primitive, then there exists a letter $i$ and a positive integer $k$ such that $\sigma^k(i)$ begins
with $i$. Considering the infinite word generated by $\sigma^k$ starting with $i$ we obtain an infinite word $u$ which is periodic
for $\sigma$. Its language  only depends on $\sigma$ and is denoted  as $L_\sigma$. Similarly, we let $X_\sigma$ denote the shift associated with $\sigma$.
 It can be defined by
\[
L_\sigma = \bigcap_{n \in \NN} \overline{\sigma^n(\AAA)}
\]
where $\overline{M}$ is the \emph{factorial closure} of $M$ (i.e., the smallest factorial language that contains
$M$).

\smallskip

As a second class of shift spaces, we consider shifts of finite type (see e.g.~\cite{Lind}) and the more
general topological Markov shifts (see e.g.~\cite{Sarig,BuzziSarig}). Let $G = (V,E)$ be a directed graph
(with possible multiple edges), with vertex set $V$ and edge set $E$. Given an edge $e \in E$, we let
$s(e)$ denote the \emph{source} of $e$, that is, its initial vertex, and $r(e)$ its \emph{range}, that is,
its final vertex. The set of infinite paths in $G$ is
\[
\Sigma_G = \{\gamma =(\gamma_n)_n \in E^\NN\mid \forall n \geq 0,\, r(\gamma_n) = s(\gamma_{n+1})\}.
\]
The space $\Sigma_G$ is naturally a subshift of $E^\NN$. The action of $T$ on the path $\gamma$ simply forgets
the first edge $\gamma_0$ of the path. If $V$ and $E$ are finite, $\Sigma_G$ is a \emph{shift of finite type},
otherwise, it is a \emph{topological Markov shift}. In the former case, one can encode the graph $G$ through
its adjacency matrix $M$ where $M_{i,j}$ is the number of arrows from $i$ to $j$.

\smallskip

There exist deep relations between substitutive dynamical systems and shifts of finite type
and we refer to~\cite{can_sie,can_sie1} and~\cite[Chap. 7]{pyt} for the details. For recognizability
reasons highlighted in~\cite{can_sie}, we will work with the two-sided subshift $X_{\sigma}$. Let
$\sigma$ be a primitive substitution and let $M$ be its incidence matrix. In order to describe the infinite words that belong
to the two-sided subshift $X_{\sigma} \subset \AAA^\ZZ$ generated by $\sigma$, one uses the following
decomposition, called Dumont-Thomas prefix-suffix decomposition \cite{DT89}.
This decomposition is closely related to the so-called Dumont-Thomas numeration (see~\cite{DT89} and the
survey~\cite{BBLT} for more details). Any biinfinite word $u \in X_\sigma$ can be written as 
\begin{equation}\label{eq:ps}
u=\cdots\, \sigma^3(p_3)\, \sigma^2(p_2)\, \sigma(p_1)\, p_0 \cdot a_0\,  s_0\,  \sigma(s_1)\, \sigma^2(s_2)\, \sigma^3(s_3)\, \cdots
\end{equation}
where $a_0=u_0$, and for each $n$, $\sigma(a_{n+1})=p_n a_n s_n$, with $p_n,s_n \in \AAA^*$. Sequences
$((p_n,a_n,s_n))_{n \in \NN}$ are infinite paths in the prefix-suffix automaton defined as the following directed graph:
its states are the letters of the alphabet, and there exists an edge between $a$ and $b$, labeled by $(p,a,s)$
if $\sigma(b)$ can be written as $\sigma(b)=pas$. In other words, it is an infinite path in the graph associated
with the incidence matrix $M$.

Such a decomposition induces a map $X_\sigma \rightarrow \Sigma_M$ between the shift $X_\sigma$ of the substitution
and the subshift of finite type $\Sigma_M$ built from the incidence matrix  $M$ of $\sigma$.
This map is onto and one-to-one up to a countable number of points (these are precisely the infinite paths for which
the corresponding sequence of prefixes or suffixes is ultimately empty). The shift map $T$ on $X_\sigma$ is then
conjugate to an adic transformation on the subshift $\Sigma_M$ whereas the action of $\sigma$ on $X_\sigma$ is
conjugate to the shift map on $\Sigma_M$.

\subsection{Recurrence, return words and minimality} \label{subsec:recurrence}
An infinite word $u$ is said to be \emph{uniformly recurrent} if every factor of $u$ occurs infinitely
often and with bounded gaps. Note that this is a property of the language $L_u$. It can be shown
that $L_u$ is uniformly recurrent if and only if $X_u$ is \emph{minimal}, that is,  $\emptyset$ and $X_u$
are the only closed shift-invariant subsets of $X_u$. For more details, see \cite{Queffelec:10}.  

The \emph{recurrence function} $R(n)$ of a uniformly recurrent word $u$ is defined as
follows: for any $n$, $R(n)$ is equal to the smallest positive integer $k$ for which
every factor of size $k$ of $u$ contains every factor of size $n$. An infinite
word $u$ is said to be \emph{linearly recurrent} if there exists a constant
$C$ such that $R(n) \leq Cn$, for all $n$. The recurrence function is similarly defined for a factorial
language which is uniformly recurrent.

There is another version of the recurrence function which may be useful in other
contexts. Let $L$ be a factorial
language which is uniformly recurrent. A \emph{return word} of an element
$w$ in $L$ is a word $v$ that starts with $w$ such that $vw$ belongs to $L$
and has exactly two occurrences of $w$ (one at the beginning and one at the end).
Let $R'(n)$ be the maximum length of a return word of an element of length $n$ in $L$.
Then we have
\[
R(n) - n \leq R'(n) \leq R(n) - n + 1.
\]
In particular $R(n)$ is linear if and only if $R'(n)$ is also linear.

\begin{thm} \label{thm:subs_syst_are_linearly_recurrent}
Let $\sigma$ be a primitive substitution and $X_\sigma$ the associated shift.
Then, $X_\sigma$ is linearly recurrent.
\end{thm}

\begin{proof}
  We first provide an upper bound for $R(2)$ and then prove how to use it to bound $R(n)$.

By taking possibly  a power, we may assume that $\sigma$ is positive. Let
  \[
  \beta_n^- = \min_{i \in \AAA} |\sigma^n(i)|
  \quad \text{and} \quad
  \beta_n^+ = \max_{i \in \AAA} |\sigma^n(i)|.
  \]
  By positivity, Perron-Frobenius theorem applies and the dominant eigenvalue of
  the incidence matrix $M_\sigma$ is positive and simple. Using once again  positivity,
  we have that all coefficients of $M_\sigma^n$ have roughly the same size. More precisely,
  there exists a constant $C$ such that
  \begin{equation} \label{eq:growth_matrix}
	C^{-1} \lambda^n \leq \beta^-_n \leq \beta^+_n \leq C \lambda^n.
  \end{equation}

  Let $w$ be a word of length $2$ in $L_\sigma$.
  There exists an image $\sigma^k(i)$, with $k \geq 0$ and $i \in \AAA$, that contains $w$.
Since $\sigma$ is positive, $w$ occurs in all  the images  of letters by $\sigma^{k+1}$.
  Now, each word of length greater than $2 \beta_{k+1}^+$ contains the image of  a letter by $\sigma^{k+1}$.
Since  there are finitely many words of length $2$, there is a $k$ which works for all
  of them and we have $R(2) \leq 2 \beta_{k+1}^+$. Remark that the same argument applied
  to words of length $n$ proves that the language $L_\sigma$ is uniformly
  recurrent or equivalently that the shift $X_\sigma$ is minimal.

  Now, we fix $n$ and provide an upper  bound for $R(n)$.
  Let $k$ be the smallest integer such that $\beta_k^- \geq n$.
  We have from~\eqref{eq:growth_matrix} that $k \leq \log n / \log \lambda + C'$ where $C' = \log C / \log \lambda + 1$.
  Now each word of length $n$ appears inside a word of the form $\sigma^k(i) \sigma^k(j)$ where $ij$ is a word of length $2$ in $L_\sigma$.
  We get that if $w$ contains all words of length $2$, then $\sigma^k(w)$ contains all words of length $n$.
  But $|\sigma^k(w)| \leq \beta_k^+ |w|$ from which we get
  \[
  R(n) \leq 2 R(2) \beta_k^+ \leq C'' n,
  \]
  with $C'' = 2 R(2) C \lambda^{C'}$,  with  the last inequality   being deduced from~\eqref{eq:growth_matrix}.
\end{proof}

Let us stress that the situation is really different for shifts of finite type, with respect to  minimality.
If the graph that determines a shift of finite type has several loops,
then it contains several periodic words (that are obtained by considering
the infinite paths that wrap  around a loop). Each periodic word has
a finite orbit under the shift $T$ which is hence a closed invariant
subset.

\subsection{A measure of chaos: factor complexity and entropy} \label{subsec:comp}
The (\emph{factor}) \emph{complexity function} of an infinite word $u$ counts the number of distinct factors of a
given length. We let $p_u(n)$ denote the number of factors of length $n$ of $u$.
The number $p_u(n+1)/p_u(n)$ is the mean number of ways to extend a word of length $n$.
Complexity function is hence a measure of predictability.

Here, we will focus on low complexity words. An infinite word $u$ is said to have
\emph{linear factor complexity} if $p_u(n)= \Theta(n)$, i.e., there exists
$C$ such that $C^{-1} n \leq p_u(n) \leq Cn$, and \emph{quadratic factor complexity} if $p_u(n)=\Theta(n^2)$. We say
that the word $u$ has \emph{at most linear} or \emph{at most quadratic complexity} if we have respectively 
$p_u(n) = O(n)$ (i.e.,  there exists $C$ such that $p_u(n) \leq Cn$), and $p_u(n) = O(n^2)$.

This notion of complexity extends in an immediate way to the language of a symbolic dynamical system $(X,T)$.
The exponential growth rate of the factor complexity $p_X(n)$ is equal to the topological entropy $h_X$ of the
system, that is,
\[
h_X = \lim_{n \to \infty} \frac{\log p_X(n)}{n} = \lim_{n \to \infty} \log \frac{p_X(n+1)}{p_X(n)}.
\]
The convergence above follows from the submultiplicativity property of the complexity function for factorial languages: $p_X(m+n) \leq p_X(m) p_X(n)$.
The complexity function is thus a refinement
of the notion of topological entropy, in particular when it has a low order of magnitude.
Note that languages with polynomial complexity have zero entropy, thus entropy does not distinguish between  them.
But, contrarily to the entropy which is invariant under measurable conjugation, the complexity function is not.
Nevertheless, if the symbolic systems $X$ and $Y$ are topologically conjugate (i.e.,  there exists a shift
invariant homeomorphism $X \rightarrow Y$),  then there exists a constant $C$ such that, for all $n$,
one has $p_X(n-c) \leq p_Y(n) \leq p_X(n+c)$, such as highlighted in~\cite{frank,Ferenczicomp}. This relies on the fact that  topological conjugacy 
is  equivalent to the existence of a sliding  block-code between both systems.
In particular, if $p(n)$ is polynomial, then both the exponent and the coefficient
of its leading term are preserved by topological conjugacy, whereas if $p(n)$ growths exponentially, then only
the entropy makes sense. Observe also that if $\sigma$ is a (non-erasing) substitution, then
$\displaystyle p_{\sigma(u)}(n) \leq \max_{a \in \AAA} |\sigma(a)| \  p_u(n)$, see \cite[Chap.4, Lemma~4.6.7]{CANT}.

We have the following elementary result.
\begin{thm} \label{thm:linear_rec_implies_linear_comp}
  Let $L$ be a language, $R$ its recurrence function and $p$ its complexity function.
  Then, $p(n) \leq R(n)$. In particular, linearly recurrent languages have at most linear
  complexity functions.
\end{thm}
Note that, as a corollary, and  according also to Theorem \ref{thm:subs_syst_are_linearly_recurrent}, we get that a shift associated with a  primitive substitution has at
most linear complexity. More generally, purely
morphic words have a low factor complexity~\cite{Pansiot,CANT}: it is at most quadratic.
See also~\cite{Deviatov} for a study of the complexity function of some morphic words.

We warn the reader that there is no bound on $R(n)$ in terms of $p(n)$ as there exist Sturmian
words (of complexity $p(n)=n+1$) for which $R(n)$ grows very fast~\cite{MH2}.

\smallskip

The situation is completely different for shifts of finite type: they generally have positive entropy.
More precisely, let $M$ be a primitive matrix and consider the associated subshift of finite type.
Then, the complexity of the shift space grows exponentially with exponent the Perron-Frobenius eigenvalue of $M$.
We refert to~\cite{Lind} for a proof.

\subsection{Frequencies and invariant measures}\label{subsec:fm1}
Let $u$ be an infinite word. The \emph{frequency} of a letter $i$ in $u$ is defined as the limit when
$n$ tends towards infinity, if it exists, of the number of occurrences of $i$ in $u_0 u_1 \cdots u_{n-1}$ divided
by $n$.  The vector $f$ whose components are given by the
frequencies of the letters  is called the \emph{letter frequency vector}. The infinite word $u$ has \emph{uniform letter frequencies} if, for every letter $i$ of $u$, the number of
occurrences of $i$ in $u_k\cdots u_{k+n-1}$ divided by $n$ has a limit when $n$ tends to infinity, uniformly in $k$.

Similarly, we can define  the frequency   and the uniform frequency of a factor, and we say that $u$ has \emph{uniform frequencies} if all its
factors have uniform frequency.

An infinite word $u \in \AAA^\NN$ is said to be $C$-\emph{balanced} if for any pair $v,w$ of factors of the
same length of $u$, and for any letter $i \in \mathcal{A}$, one has $||v|_i - |w|_i| \leq C$. It is said
\emph{balanced} if there exists $C>0$ such that it is $C$-balanced.
\begin{prp}\label{prop:balanced}
An infinite word $u \in \AAA^\NN$ is balanced if and only if it has uniform letter frequencies and there exists
a constant $B$ such that for any factor $w$ of $u$, we have $||w|_i - f_i |w|| \leq B$ for all letter $i$ in $\AAA$,
where $f_i$ is the frequency of $i$. 
\end{prp}

\begin{proof}
Let $u$ be an infinite word with letter frequency vector $f$ and such that $||w|_i - f_i |w|| \leq B$ for every factor
$w$ and  every letter $i$ in $\AAA$. Then, for every pair of factors $w_1$ and $w_2$ with the  same length $n$, we have by
triangular inequality
\[
||w_1|_i - |w_2|_i| \leq ||w_1|_i - n f_i| + ||w_2|_i - n f_i| \leq 2B.
\]
Hence $L$ is $2B$-balanced (see also \cite[Proposition 7]{Adam0}).

Conversely, assume that $u$ is $C$-balanced for some $C>0$. We fix a letter $i \in \AAA$.
For every non-negative integer $p$, let $N_p$ be defined as an integer $N$ such that
for every word of length $p$ of $u$, the number of occurrences of the letter
$i$ belongs to the set $\{N,N+1,\cdots, N+C\}$. 
 
We first observe that the sequence $(N_p/p)_{p \in \NN}$ is a Cauchy sequence.
Indeed consider a factor $w$ of length $pq$, where $p,q \in {\mathbb N}$.
The number $|w|_i$ of occurrences of $i$ in $w$ satisfies
\[
p N_q \leq |w|_i \leq p N_q + p C, \ \quad q N_p \leq |w|_i \leq q N_p + q C.
\]
We deduce that $-q C \leq q N_p -p N_q \leq p C$ and thus $-C \leq N_p - p N_q /q \leq p C / q$.
 
Let $f_i$ stand for $\lim_q N_q/q$. By letting $q$ tend to infinity, one then deduces that 
$-C \leq N_p -p f_i \leq 0$. Thus, for any factor $w$ of $u$ we have
\[
\left|\frac{| w|_i }{|w| }- f_i \right|\leq \frac{C}{|w|},
\]
which  was to be proved.
\end{proof}

Note that  having frequencies is a property of an infinite  word while  having uniform frequencies is a property of the associated
language or shift. In particular, it makes sense to speak about uniform frequencies for languages or shifts. A probability measure $\mu$ on $X$ is
said invariant if $\mu(T^{-1} A) = \mu(A)$ for every measurable subset $A \subset X$.
An invariant probability measure on a shift $X$ is said \emph{ergodic} if any shift-invariant measurable set
has either measure $0$ or $1$ (which can be thought as a generalization of a situation where Kolmogorov 0-1 law holds).
If $\mu$ is an ergodic measure on $X$, then we know from  the Birkhoff  ergodic Theorem that $\mu$-almost every infinite word in $X$ has
frequency $\mu([w])$, for any factor $w$, where $[w] = \{u \in X; u_0 \ldots u_{n-1} = w\}$ is a
\emph{cylinder},  but this frequency is not necessarily uniform.
If $X$ is \emph{uniquely ergodic}, i.e., there
exists a unique shift-invariant probability measure on $X$, then  the unique invariant measure on $X$ is ergodic and the convergence
in the Birkhoff  ergodic Theorem holds for every infinite word in $X$. The property of  having uniform  factor frequencies  for a shift $X$ is actually equivalent to  unique ergodicity.  In that case, one
can recover the frequency of a factor $w$ of length $n$ as $\mu([w])$.  Balancedness is again a property of the associated shift and may be thought as a strong form of uniform letter frequencies. For more details on invariant measures and ergodicity, we refer to~\cite{Queffelec:10} and~\cite[Chap. 7]{CANT}.

Lastly, let us stress the fact that linear recurrence is related both to low factor complexity (by Theorem \ref{thm:linear_rec_implies_linear_comp})
and  to unique ergodicity, such as illustrated by the following result.

\begin{thm}[Linear recurrence \cite{DHS99}]\label{thm:LR}
Let $X$ be a linearly recurrent symbolic dynamical system. Then, $(X,T)$ is uniquely ergodic.
\end{thm}

\smallskip

For subshifts of finite type, there are many invariant measures, but the \emph{pressure principle} provides a unique measure of maximal entropy (and more generally a unique measure associated with any weight function) \cite{Walters,Lind}. The measure of maximal entropy is built as follows. Consider an edge shift of finite type $\Sigma_M$ associated with the adjacency matrix $M$ of its underlying graph. If the matrix $M$ is primitive, then the subshift of finite type $\Sigma_M$ has a unique measure of maximal entropy, called the \emph{Parry measure}, and built from the Perron-Frobenius eigenvector of the matrix $M$. Indeed, let $M$ be a primitive matrix and let $\lambda$ and $p$ be respectively the Perron-Frobenius eigenvalue and an associated right eigenvector of $M$. Then, we associate with each edge $ij$ the weight $\frac{p_j}{\lambda p_i}$. The associated Markov measure is invariant for the shift on $\Sigma_M$; this is the Parry measure. It is possible to build many further invariant measures on $\Sigma_M$. Some of great importance are obtained by considering potentials (see e.g.~\cite{Lind}).

\smallskip

We know that  the  substitutive subshift $X_\sigma$ determined by a primitive substitution $\sigma$ is  uniquely
ergodic from Theorem~\ref{thm:subs_syst_are_linearly_recurrent} and \ref{thm:LR} . We provide another
proof of that result in Theorem~\ref{thm:invariant_measures}. Indeed all the information about the unique shift-invariant
measure for $X_\sigma$ can actually be obtained from the incidence matrix $M_\sigma$.
First of all, the Perron-Frobenius eigenvector of $M_\sigma$ is the  letter  frequency vector  of $X_\sigma$.
To obtain information for  word frequencies,  one then  needs the Dumont-Thomas prefix-suffix decomposition
recalled in Section~\ref{subsec:subs_and_finite_type}. It induces a map
$X_\sigma \rightarrow \Sigma_M$ between the shift $X_\sigma$ of a primitive substitution
$\sigma$ and the subshift of finite type $\Sigma_M$ built from the incidence matrix of $\sigma$.
The unique invariant measure on $X_\sigma$ is mapped on the Parry measure on $\Sigma_M$.
Note that  a random word in $X_\sigma$ is obtained by following a random walk in $\Sigma_M$ according to the Parry measure.
We will see that, in the $S$-adic framework, as the substitutions change at each step,
the infinite paths,   considered in the prefix-suffix automaton for substitutions, are seen on a Bratteli diagram (see Remark~\ref{rem:Bratteli}, p.~\pageref{rem:Bratteli}).

\begin{rem}\label{rem:Bosh}
Note that the notion of factor complexity discussed in the previous section also provides
information concerning the number of ergodic measures. 
As shown by M.~Boshernitzan~\cite{Bosh84}, the topological quantity $\liminf p_X(n)/n$
provides a bound for the number of ergodic measures; for more details, see~\cite[Chap. 7]{CANT}.
Furthermore, for minimal shifts, $\limsup p_X(n)/n <3$ implies unique ergodicity~\cite{Bosh84}.
This can be applied for instance to Arnoux-Rauzy words on a three-letter alphabet (they have factor complexity $2n+1$).
Note also that there exist examples of infinite words with exponential factor complexity
(and arbitrarily high entropy) and uniform frequencies~\cite{Grillenberger72}.
\end{rem}

\section{$S$-adic words and families}\label{sec:representation}
Now that we have recalled the main definitions and properties concerning symbolic dynamical systems, and substitutive ones in particular, we introduce the corresponding notions for $S$-adic words and associated symbolic systems.

\subsection{First definitions}
Let $S $ be a set of (non-erasing) morphisms of the free monoid. 
Let $s = (\sigma_n)_{n\in\NN} \in S^\NN$, with $\sigma_n: \AAA_{n+1}^* \rightarrow \AAA_n^*$, be a sequence of  substitutions, and let $(a_n)_{n\in\mathbb{N}}$ be a sequence of letters with $a_n \in \AAA_n$ for all $n$. We say that the infinite word $u \in \AAA^\NN$ admits $((\sigma_n,a_n))_n$ as an \emph{$S$-adic representation} if
$$u = \lim_{n\to\infty} \sigma_0 \sigma_1 \cdots \sigma_{n-1}(a_n).$$ The sequence $s$ is called the~\emph{directive sequence} and the sequences of letters $(a_n)_n$ will only play a minor role compared to the directive sequence. If the set $S$ is finite, it makes no difference to consider a constant alphabet (i.e., $\AAA_n^* = \AAA^*$ for all $n$ and for all substitution  $\sigma$ in $S$). 
 As we will constantly use products of   substitutions, we introduce the notation $\sigma_{[k,l)}$ for the product $\sigma_k \sigma_{k+1} \ldots \sigma_{k+l-1}$. In particular, $\sigma_{[0,n)} = \sigma_0 \sigma_1 \ldots \sigma_{n-1}$.

A word admits many possible $S$-adic representations. But some $S$-adic representations might be
useful to get information about the word. More precisely, some properties are actually equivalent
to have some $S$-adic representation of a special kind (see in particular Theorem~\ref{thm:minimality}
for minimality and Theorem~\ref{thm:linearly_recurrent} for linear recurrence).

Of course, an infinite word that is generated by a substitution is an example of an $S$-adic word
and Theorem~\ref{thm:subs_syst_are_linearly_recurrent} may be seen as a prototype of an $S$-adic result. 

In order to avoid degeneracy construction we introduce the following definition.
\begin{dfn}[Everywhere growing]\label{def:grow}
 An $S$-adic representation defined by the directive sequence
 $(\sigma_n)_{n \in \NN}$, where $\sigma_n: \AAA_{n+1}^* \rightarrow \AAA_n^*$,
 is \emph{everywhere growing} if for any sequence of letters $(a_n)_n $, one has
 \[
 \lim _{n \rightarrow +\infty} | \sigma_{[0,n)}(a_n)|= +\infty.
 \]
\end{dfn}
We define the $S$-adic language $L$ associated with the $S$-adic   directive sequence $(\sigma_n)_n$ assumed to be  everywhere growing  as 
\[
L = \bigcap_{n \in \NN} \overline{\sigma_0 \sigma_{1} \ldots \sigma_{n-1} (\AAA_n)},
\]
where, as in Section~\ref{subsec:subs_and_finite_type}, $\overline{M}$ is the factorial closure of $M$.
We also define the shift associated with $(\sigma_n)_n$ as the shift associated with $L$.

\begin{exa}[Sturmian words I]\label{ex:sturm1}
We define Sturmian words for which the Fibonacci word of Example~\ref{ex:fibonacci} p.~\pageref{ex:fibonacci}
is a particular case. We consider the substitutions $\tau_a$ and $\tau_b$ defined over
the alphabet $\AAA = \{a,b\}$ by 
$\tau_a \colon a \mapsto a, b \mapsto ab$ and $\tau_b \colon a \mapsto ba, b \mapsto b$.
Let $(i_n) \in \{a,b\}^\NN$.
The following limits
\begin{equation}\label{eq:sturm}
 u = \lim_{n\to\infty} \tau_{i_0 }\tau_{i_1} \cdots \tau_{i_{n-1}}(a) = \lim_{n\to\infty} \tau_{i_0} \tau_{i_1} \cdots \tau_{i_{n-1}}(b)
\end{equation}
exist and coincide whenever the directive sequence $(i_n)_n$ is not ultimately constant
(it is easily shown that the shortest of the two images by  $\tau_{i_0} \tau_{i_1} \ldots \tau_{i_{n-1}}$ is a prefix of the other).
This latter condition is equivalent to the everywhere growing property of Definition~\ref{def:grow}.
The infinite words thus produced belong to the class of Sturmian words.
More generally, a \emph{Sturmian word} is an infinite word whose set of factors coincides with the set of factors of
a sequence of the form~\eqref{eq:sturm}, with the sequence $(i_n)_{n\geq 0}$ being not ultimately constant.

Let us consider a second set of substitutions $\mu_a \colon a \mapsto a, b \mapsto ba$ and $\mu_1 \colon a \mapsto ab, b \mapsto b$.
In this latter case, the two corresponding limits do exist, namely 
\begin{equation} \label{eq:sturm_bis}
w_a = \lim_{n\to\infty} \mu_{i_0} \mu_{i_1} \cdots \mu_{i_n}(a)
\quad \text{and} \quad
w_b = \lim_{n\to\infty} \mu_{i_0} \mu_{i_1} \cdots \mu_{i_n}(b)
\end{equation}
for any sequence $(i_n) \in \{a,b\}^{\mathbb N}$. As $w_a$ starts with $a$ and $w_b$ with $b$ they do not coincide.
But, provided the directive sequence $(i_n)_n$ is not ultimately constant, the languages generated by $w_a$ and $w_b$ are the same,
and they also coincide with  the language generated by $u$ (for the same sequence $(i_n)$).
\end{exa}
\begin{exa}[Arnoux-Rauzy words]\label{exa:AR}
Let $\AAA=\{1,2,\ldots,d\}$. We define the set of Arnoux-Rauzy substitutions as $S_{AR} = \{\mu_i \mid i \in \mathcal{A}\}$ where
\[
\mu_i:\ i \mapsto i,\ j \mapsto ji\ \mbox{for}\ j \in \mathcal{A} \setminus \{i\}\,.
\]
Example~\ref{ex:sturm1} corresponds to the case $d=2$. An \emph{Arnoux-Rauzy word}~\cite{AR91} is an infinite word $\omega \in \mathcal{A}^\mathbb{N}$ whose set of factors coincides with the set of factors of a sequence of the form 
\[
\lim_{n\to\infty} \mu_{i_0} \mu_{i_1} \cdots \mu_{i_n} (1), 
\]
where the sequence $(i_n)_{n\geq 0} \in \mathcal{A}^\mathbb{N}$ is such that every letter in~$\AAA$ occurs infinitely often in~$(i_n)_{n\ge0}$. 
For more on Arnoux-Rauzy words, see~\cite{CFM08,CFZ00}.
\end{exa}

\begin{exa}[Three interval exchange words] \label{exa:3iet}
An interval exchange transformation is a map of the interval obtained as follows: cut the interval into pieces and reorder
them with respect to some permutation. Such a map comes with a natural coding which consists in giving a name
to each subinterval. Let us use the abbreviation $k$-iet for an interval exchange transformation with $k$ subintervals.
A $k$-iet word is by definition a word that is the natural coding of an orbit of a $k$-iet.
Rotations (for which the coding gives Sturmian words) are exactly the $2$-iet. Here, we represent  3-iet words
in an $S$-adic way and we refer to~\cite{Viana,Yoccoz}
for more details on the theory of interval exchanges.
In order to obtain an $S$-adic representation of an interval exchange transformation, one can use the so-called
\emph{Rauzy induction} introduced in~\cite{ra1}. For 3-iet words one obtains the following
\begin{center}
\includegraphics{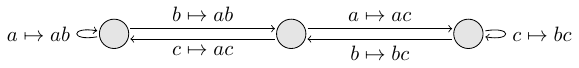}
\end{center}
Each edge is labeled by a substitution on $\{a,b,c\}$ where we omit letters that are mapped to themselves. We say that the letter $i$ \emph{wins} on a given edge if the substitution is of the form $i \mapsto ix$ or $i \mapsto xi$. Any infinite path in the above graph such that all letters win infinitely often gives a sequence of substitutions whose associated limit word is  a (natural) coding of a minimal 3-iet. 
For $S$-adic representations for 3-iet see~\cite{Adam2,Santini}, and for another type of induction for 3-iet see also~\cite{FerRocha,FHZ1,FHZ2,FHZ3}. Lastly, for $k$-iet, see~\cite{ra1,Ferzam}.

More generally, given a dynamical system that may be represented by a cutting and stacking construction (i.e., a Bratteli digram), the sequence of cuttings and stackings performed  provides naturally a sequence of substitutions that describes the orbits of this dynamical system (see~\cite{DHS99}).\end{exa}
 
\begin{exa}[$S$-adic expansions by return words]\label{ex:return}
Let $u$ be a uniformly recurrent word on $\AAA = \AAA_0$. Let $w$ be a non-empty factor
of $u$. Recall from Section~\ref{subsec:recurrence} that a return word of $w$ is a word
separating two successive occurrences of the word $w$ in $u$ (possibly with overlap). By
coding the initial word $u$ with these return words, one obtains an infinite word called
the \emph{derived word}, defined on a finite alphabet, and still uniformly recurrent. Indeed, start with  the letter $u_0$. There exist  finitely many  return words to $u_0$. Let $ w_1$, $w_2, \ldots, w_{d_1}$ be these return words, and consider the associated morphism $\sigma_0: \AAA_1 \rightarrow \AAA_0, \ i \mapsto w_i$, with $\AAA_1=\{1,\ldots,d_1\}$. Then, there exists a unique word $u'$ on $\AAA_1$ such that $u = \sigma_0(u')$. Moreover, $u'$ is uniformly recurrent. It is hence possible to repeat the construction and one obtains an $S$-adic representation of $u$. We refer to \cite{DHS99}, \cite{Dur03} and \cite{Leroy} for the details. The alphabets of this representation are a priori of unbounded size. In the particular case where
$u$ is a primitive substitutive word, then the set of derived words is finite. This is even a characterization (see \cite{Dur98}): a uniformly recurrent word is substitutive if and only if the set of its derived words is finite. 
\end{exa}

\begin{rem} \label{rem:Bratteli}
Let us briefly review one variation around the definition of $S$-adic words. Recall that $T$ denotes the shift on infinite sequences.
We extend to finite words by mapping a finite word of length $n$ to its suffix of length $n-1$. The map $T$ is continuous on $\AAA^* \cup \AAA^\NN$. We will abuse the notation in the sense that $T$ stands for the shift on the sets of words with different alphabets. One can define more general $S$-adic words by allowing products of the form
\[
T^{k_0}\, \sigma_0\, T^{k_1}\, \sigma_1\, \ldots\, T^{k_{n-1}}\, \sigma_{n-1} (a_n).
\]
These representations may be thought as a version of the Dumont-Thomas prefix-suffix decomposition (generalizing~\eqref{eq:ps})
\begin{equation} \label{eq:adic_DT}
\cdots\, \sigma_{[0,2)}(p_2)\, \sigma_{[0,1)}(p_1)\, p_0\, .\, a_0\, s_0\, \sigma_{[0,1)}(s_1)\, \sigma_{[0,2)}(s_2)\, \cdots
\end{equation}
where we keep only the right infinite word and $k_i = |p_i|$. This construction is especially useful as it may be used to describe all infinite words that belong to the language generated by the directive sequence $(\sigma_n)_n$. Note that possible sequences $((p_n,a_n,s_n))_n$ are not defined as paths on a fixed directed graph as it was the case for purely substitutive sequences (see the end of Section~\ref{subsec:subs_and_finite_type}). Nevertheless a similar construction can be done by building an infinite acyclic graph, namely a Bratteli diagram, as follows. The set of vertices of the graph is the disjoint union $\{\emptyset\} \sqcup \AAA_0 \sqcup \AAA_1 \sqcup \ldots$. Then, we put an edge between $\emptyset$ and each letter in $\AAA_0$, and an edge between a letter $a$ in $\AAA_n$ and $b$ in $\AAA_{n+1}$ if $\sigma_n(b) = pas$ for some words $p$ and $s$. With each infinite path in this diagram, provided that neither the sequence of prefixes nor the sequence of suffixes is ultimately empty, is associated the well-defined biinfinite word defined by~\eqref{eq:adic_DT}.
\end{rem}
We develop the case of Sturmian words in the Example~\ref{ex:sturm2} below (see also~\cite{ice}).

\begin{exa}[Sturmian words II]\label{ex:sturm2} We continue Example~\ref{ex:sturm1}.
Since the incidence matrices of the $\tau_i$'s (or of the $\mu_i$'s) generate $GL(2,{\mathbb N})$, any Sturmian language can be generated
with either $\{\tau_0,\tau_1\}$ or $\{\mu_0,\mu_1\}$ with the same directive sequence (seen as a sequence $(i_n)_n$ on $\{0,1\}$). More precisely, for any Sturmian language $L_{\alpha}$, there exists a unique infinite word $u$ of the form~\eqref{eq:sturm} (or $w_0$ or $w_1$ given by~\eqref{eq:sturm_bis}) whose language coincides with $L_{\alpha}$. Nevertheless, not every Sturmian word admits an $S$-adic expansion obtained by using either $\{\tau_0,\tau_1\}$ or $\{\mu_0,\mu_1\}$. Indeed, there are two ways of obtaining all possible Sturmian words. The first one consists in considering the set of substitutions $\{\tau_0,\tau_1,\mu_0,\mu_1\}$ (and even, only three substitutions are required). The other one, mentioned above, uses shifts. More precisely, every Sturmian word $u$ (on the alphabet $\{0,1\}$) admits a unique $S$-adic representation as an infinite composition of the form
\[
u = T^{c_0}\, \tau_0^{a_0}\, T^{c_1}\, \tau_1^{a_1}\, T^{c_2}\, \tau_0^{a_2}\, T^{c_3}\, \tau_1^{a_3}\, \cdots,
\]
where $T$ is the shift map, $a_k \geq c_k \geq 0$ for all $k$ and $a_k \geq 1$ for $k \geq 1$, and if $c_k=a_k$, then $c_{k-1}=0$.
The sequence $(a_k)_{k}$ turns out to be the sequence of partial quotients of the continued fraction associated with the slope $\alpha$ of $u$ (which can also be defined as the frequency of the letter $1$ in $u$), while $(c_k)_{k}$ is the sequence of digits in the arithmetic Ostrowski expansion of the intercept of $u$. When generating any Sturmian word using only $4$ substitutions and no shifted images, similar admissibility conditions  also occur. For more details, see \cite{afh,ice,LRi}, \cite[Chap. 2]{Lothaire:2002} and \cite[Chap. 5]{pyt}.

We insist on the fact that, from the viewpoint of languages or shifts, the occurrence of shifted images does not matter. 
\end{exa}

\begin{rem}\label{rem:cass}
To be `$S$-adic' is not an intrinsic property of an infinite word, but a way to construct it. Indeed, without further restriction, every infinite word $u=(u_n)_n \in \AAA^\NN$ admits an $S$-adic representation. Here, we recall the classical construction due to J. Cassaigne (see \cite{pyt2}). We consider $u=(u_n)_n \in \AAA^\NN$. We introduce a further letter $\ell \not \in \AAA$, and we work on $\AAA \cup \{\ell\}$.
 For every letter $a \in \AAA$, we introduce the substitution $\sigma_{a}$ defined over the alphabet $\AAA \cup \{\ell\}$ as
$\sigma_a(b)= b$, for $b \in \AAA$, and $\sigma_a (\ell)= \ell a$. We also consider the substitution $\tau_{\ell}$ over the alphabet $\AAA \cup \{\ell\}$ defined as $\tau_{\ell}(\ell )= u_0$, and $\tau_{\ell}(b)= b$, for all $b \in \AAA$. One checks that
\[
u = \lim _{n \to +\infty} \tau_{\ell} \, \sigma_{u_1} \, \sigma_{u_1} \, \ldots \, \sigma_{u_n} (\ell).
\]
Hence $u$ admits an $S$-adic representation with $S = \{\sigma_a\mid a \in {\mathcal A}\} \cup \{\tau_{\ell}\}$.
We stress the fact that, despite $u$ belongs to $\AAA^\NN$, this $S$-representation involves the larger-size alphabet $\AAA \cup \{\ell\}$. Observe also that this representation is not everywhere growing.
\end{rem}

\begin{prp}[\cite{pyt2}]  \label{prp:power_criterion}
  Let $L$ be a factorial extendable language on $\{0,1\}$ that contains
  $10^k1$ and $01^k0$ for all $k \in \NN$.
  Then, $L$ admits no $S$-adic expansion  for which  all the substitutions are defined on a 2-letter
  alphabet.
\end{prp}
In particular this proposition applies to the well-known Champernowne word obtained by concatenating
the representations in base $2$ of the non-negative integers
\[
0\,1\,10\,11\,100\,101 \,110 \,111\, \cdots
\]
The language generated by the Champernowne word is $\{0,1\}^*$ and hence of complexity $p(n) = 2^n$. It is
nevertheless possible to build example of quadratic complexity by concatenation of the powers $0^k$ and
$1^k$ as follows:
\[
0\, 1\, 00\, 11\, 000\, 111\, 0000\, 1111\, \cdots
\]
One can check that this language has complexity $n(n+1)/2 + 1$. It is even possible to build morphic examples.

Let us also observe  that such a language $L$ is never uniformly recurrent. But it is possible to build an
infinite word $u$ on a 2-letter alphabet, with quadratic complexity, and which admits no everywhere growing
$S$-adic representation with all substitutions defined on a uniformly bounded alphabet (see Remark
 \ref{rem:no_bounded_adic}). 

Finally, let us mention that the presence of too much powers in a word forbids linear complexity~\cite[Corollary~3.24]{DLR}.

\begin{proof}
We prove it by contradiction. Assume that $(\sigma_n)_n$ is a directive sequence of such a language with all the 
$\sigma_n$ defined on $\{0,1\}$. Then, for each $n$, any word in $L$ is a factor of a word made of  a concatenation of
the two words $\sigma_{[0,n)}(0)$ and $\sigma_{[0,n)}(1)$. Since $u$ contains arbitrarily large ranges of $0's$, then
at least one of them contains only $0$'s. The same holds by considering the $1$'s. But the ranges between
two successive occurrences of $0$'s, and similarly, between two successive occurrences of $1$'s, take all
possible values. Hence the only possibility is that $\sigma_{[0,n)}$ equals the identity $0 \mapsto 0, 1 \mapsto 1$ or the 
permutation $0 \mapsto 1,\, 1 \mapsto 0$. This yields the desired contradiction.
\end{proof}

\begin{exa}[Sturmian words and Gauss continued fractions]\label{exa:sturm}
The usual continued fraction algorithm provides the best rational approximations of a real number $\alpha$ and describes in a natural way the combinatorial properties of Sturmian words, as well as those of their associated symbolic dynamical systems. Indeed Sturmian words are natural codings of rotations of the unit circle $\TT = \RR / \ZZ$. More precisely, let $\alpha \in (0,1)$ be an irrational number, and let $T_{\alpha} \colon x \mapsto x+\alpha \mod 1$ denote the rotation by $\alpha$ on the unit circle $\TT = \RR / \ZZ$. We obtain a natural coding on the alphabet $\AAA = \{a,b\}$ by considering the two-interval partitions made of $I_a = [0,1-\alpha)$ and $I_b = [1-\alpha,1)$. The words that are obtained as  natural codings of infinite orbits of $T_{\alpha}$ are exactly the Sturmian words of slope $\alpha$, and they define a shift $X_{\alpha} \subset \{0,1\}^\NN$. We have seen that the latter may be obtained as a natural $S$-adic symbolic dynamical system. Note that the parametrization of all rotations is naturally the circle $\TT = \RR / \ZZ$. But from our $S$-adic viewpoint, the rotations are parametrized by infinite sequences on the two-letter alphabet $\{0,1\}$ that encode the Gauss continued fraction transformation.
\end{exa}

\subsection{Main issues}
The questions one might want to ask concerning $S$-adic expansions are of a various nature. First, given an $S$-adic expansion (and not a word),  one has to prove the existence of a limit! This is a combinatorial (in the sense of word combinatorics) question.

Recall that primitive substitutions give rise to shifts which are minimal, uniquely ergodic and with linear complexity.
We will extend each of these properties into the $S$-adic framework. From a combinatorial viewpoint, we first might want to count the number of factors of a given length. This will be developed in Section~\ref{sec:complexity}. This question is deeper as it might occur at first view. Indeed, it was one of the main motivations for the introduction of $S$-adic expansions, in an attempt to get characterizations of infinite words with linear factor complexity function. Dynamical questions, such as the minimality and invariant measures of the associated symbolic dynamical system are then natural. This will be discussed in Section~\ref{sec:dynamics}.

We can refine the question of the unique ergodicity in a quantitative way through the study of discrepancy (or deviations of Birkhoff sums). It turns out that it is intimately related to Diophantine approximation. Namely, both the growth of discrepancies for $S$-adic systems and the Diophantine  approximation exponent for multidimensional continued fractions are related to the ratio of the two first Lyapunov exponents. Lyapunov exponents are introduced in Section~\ref{sec:Lyapunov_exponents} which is 
devoted to these  convergence quality issues.

\subsection{$S$-adic families} \label{sec:graph}
As observed in the previous examples, many $S$-adic words naturally appear in families. All of them are actually associated with a construction that we call $S$-adic graph. This section introduces these families of $S$-adic languages and how they can be studied through another dynamical system, namely the shift that operates on the set of directive sequences.

We first illustrate this approach with the family of Sturmian words. According to Example~\ref{exa:sturm}, a Sturmian shift $X_{\alpha}$ is determined by the associated rotation $x \mapsto x +\alpha$, and more precisely by its angle $\alpha$, which is a real number in $[0,1]$. The interval $[0,1]$ carries a natural measure, namely the Lebesgue measure. We thus can say that a property holds for \emph{almost all} Sturmian systems $X_{\alpha}$ if it holds for almost all ${\alpha} $ in $[0,1]$. Some properties hold for all the Sturmian systems, such as $1$-balancedness, or the fact that the complexity function equals $n+1$ for all $n$ (see e.g. \cite{Lothaire:2002}). But some other properties hold only generically. Let us recall from \cite{MH2} the following striking example of such a generic behavior. We use the notation $R_{\alpha}$ for the recurrence function of the Sturmian shift $X_{\alpha}$. Recall that $R_\alpha(n)$ is the smallest integer $k$ such that any factor in $X_\alpha$ of length $k$ contains all  the factors of length $n$.

\begin{thm}[\cite{MH2}]
Let $\Phi $ be a function which is positive and non-decreasing on ${\mathbb R}_+$ and such that $\lim_{ x \to +\infty} \Phi (x) = +\infty$.
For almost all $\alpha$, $\limsup _n \frac{R_{\alpha} (n) }{ n \Phi(\log n)} $ is finite or infinite according to the fact that the series
$\sum \frac{1}{\Phi(n)}$ is convergent or divergent.
\end{thm}
In particular, for a.e. $\alpha$
$$\limsup \frac{R_{\alpha}(n)}{ n \log n}= +\infty, \ \mbox{ and } \limsup \frac{R_{\alpha}(n)}{ n( \log n)^a}= 0 \mbox{ for } a>1. $$

The idea behind this generic phenomenon for Sturmian words 
comes from the ergodicity of the renormalization dynamics provided by the Gauss map $x \mapsto \left\{ 1/x \right\}$ on the set of slopes $[0,1]$.
 
More generally, given a minimal symbolic dynamical system $(X,T)$ generated by an $S$-adic infinite word $u$ of the form $\lim_{n\to\infty} \sigma_{[0,n)}(a_n),$ two actions enter into play: the shift acting on the symbolic dynamical system $X$, and the shift acting on the associated directive sequence $s=(\sigma_n)_n$. Note that in the case of a substitutive system, this second action is somewhat trivial as the sequence $s$ is periodic. It is also useful to consider simultaneously all the sequences $u^{(k)}= \lim_{n\to\infty} \sigma_{[k,n)} (a_n)$ in some combinatorial questions (see as an illustration~\cite{CFZ00,BCS} for estimates of (im)balance properties for Arnoux-Rauzy sequences).

For Sturmian shifts (Example ~\ref{ex:sturm1}) or Arnoux-Rauzy shifts (Example~\ref{exa:AR}) we saw that any directive sequence (up to some non-degeneracy assumption) is allowed. For interval exchanges this is not the case and the order in which the substitutions are composed is governed by a finite set of rules encoded in a graph where substitutions are on edges, and paths correspond to allowed compositions (see Example~\ref{exa:3iet}). This is the viewpoint developed below, according to~\cite{delecroix}, where the notion of $S$-adic graph is introduced. By endowing the graph $G$ with ergodic invariant measures, Lyapunov exponents can be defined (see Section~\ref{subsec:definition_Lyapunov}) which control different properties of the associated graph-directed $S$-adic systems such as the global complexity and also some properties of generic $S$-adic sequences (with respect to that measure) such as the balancedness (see Section~\ref{subsec:deviation_Birkhoff_and_balancedness}).
 
More precisely, an \emph{$S$-adic graph} is a directed strongly connected graph $G = (V,E)$ (with possibly multiple edges), and moreover, with each edge $e$ in $E$ is associated a substitution $\tau_e: \AAA^* \rightarrow \AAA^*$ where $\AAA$ denotes some fixed alphabet\footnote{We restrict the definition of $S$-adic graph to a constant finite alphabet as we do not know natural examples where unbounded alphabets do appear (but we might have an unbounded number of substitutions). Nevertheless, it would still  be possible to have an alphabet associated with each vertex of the graph with the suitable compatibility conditions for the edges.}. We let $\tau$ denote the map $e \mapsto \tau_e$, and refer to the pair $(G,\tau)$ as an \emph{$S$-adic graph}. Each edge $e \in E$ joins a vertex $s(e)$, called the \emph{source}, to a vertex $r(e)$, called the \emph{range}. With the graph $G$ we associate the edge shift $\Sigma_G \subset E^\NN$ made of all infinite paths in $G$:
\[
\Sigma_G = \{(\gamma_n) \in E^\NN \mid \forall n \in \NN, \, r(\gamma_n) = s(\gamma_{n+1}) \}.
\]

With a point $\gamma \in \Sigma_G$, we associate a sequence of substitutions $(\tau_{\gamma_i})_i$ read on the edges of the infinite path $\gamma$. The set of substitutions involved is provided by $\{\tau_e\}$. We let $X_G(\gamma)=X_{\gamma}$ denote the $S$-adic dynamical system associated with the sequence $(\tau_{\gamma_i})_i$. The \emph{graph-directed $S$-adic system} $X_G$ associated with the $S$-adic graph $(G,\tau)$ is then defined as the family of the shift paces $X_G(\gamma)$ for $\gamma \in \Sigma_G$. The $S$-adic graph $(G,\tau)$ is said to \emph{parametrize} the family $X_G$. There is no restriction on the size of the graph which may be finite or infinite.

In the framework of $S$-adic graphs, we emphasize that we have two levels of dynamics: the $S$-adic shifts obtained $X_G(\gamma)$ along infinite paths $\gamma$ in $G$, and the shift on $\Sigma_G$. Those two dynamics are dramatically opposite. The languages of $S$-adic words usually have a low factor complexity (see Section~\ref{sec:complexity}), and thus, zero entropy, whereas the shift $\Sigma_G$ is a subshift of finite type with positive entropy. In Section~\ref{sec:Lyapunov_exponents}, we will introduce suitable measures on the shift space $\Sigma_G$ which may be used to provide generic results for $S$-adic sequences built from the $S$-adic graph $(G,\tau)$.

\subsection{Multidimensional continued fractions} \label{sec:cf}
The $S$-adic framework is intimately connected with continued fractions. Indeed, finding an $S$-adic description of a symbolic dynamical system is a first step toward the description of a multidimensional continued fraction algorithm that governs its 
letter frequency vector. More precisely, given a shift $X$ on $\AAA = \{1,\ldots,d\}$, we define the cone of letter frequencies $f \in \RR^d$ as the cone generated by letter frequencies of elements in $X$ (when they exist). For $S$-adic shifts the cone of letter frequencies is intimately related to the directive sequence $(\sigma_n)_n$ as shown in the following result.
\begin{thm} \label{thm:cone_and_frequencies}
Let    $(\sigma_n)_n$, with $\sigma_n: \AAA_{n+1}^* \rightarrow \AAA_n^*$, be a sequence of substitutions
which is everywhere growing and let $X$ be 
the associated $S$-adic shift.
Let $M_n$ be the incidence matrix of $\sigma_n$. Then, the cone
\[
\lim_{n \to \infty} M_0 M_1 \ldots M_n \RR^{d_n}_+,
\]
where $d_n$ stands for the cardinality of $\AAA_n$, 
is the convex hull of the set of half lines $\RR_+ f$ generated by the letter
frequency vectors $f$ of words in $X$.
\end{thm}
For the proof of this result, see Theorem~\ref{thm:invariant_measures}.

Finding `good' $S$-adic representations of words belongs to what is called the Rauzy program in the survey~\cite{BFZ}: the idea is to find suitable extensions of the $S$-adic description of Sturmian systems thanks to regular continued fractions (see Example \ref{exa:sturm}), to other symbolic dynamical systems, with the hope of producing `good' simultaneous approximation algorithms. This program started with Arnoux-Rauzy words~\cite{AR91} and interval exchanges~\cite{ra0,ra1}, and continued e.g. with binary codings of rotations~\cite{Adam2,AD,Didier}.
For all these families, the operation of induction (first return map) can be seen as a continued fraction algorithm.

Conversely, we can decide to start with a multidimensional continued fraction algorithm and associate with it an $S$-adic system. We then translate a continued fraction algorithm into $S$-adic terms. Let us recall that in the multidimensional case, there is no canonical extension of the regular continued fraction algorithm: see~\cite{BRENTJES} or~\cite{Schweiger00} for a summary, see also the survey~\cite{Integers}. Nevertheless, most classical multidimensional continued fraction algorithms (like Jacobi-Perron or Brun) produce sequences of matrices in $GL(d,\mathbb{Z})$ as follows.

Let $X \subset \RR^d_+$ (one usually take $X = \RR^d_+$) and let $(X_i)_{i \in I}$ be a finite or countable
partition of $X$ into measurable subsets. Let $M_i$ be non-negative integer matrices so that $M_i X \subset X_i$.
From this data, we define a \emph{$d$-dimensional continued fraction map} over $X$ as the map $F:X\to X$
defined by $F(x) = M_i^{-1} x$ if $x \in X_i$. We define $M(x) = M_i$ if $x \in X_i$.

The associated \emph{continued fraction algorithm} consists in iteratively applying the map $F$ on a vector $x \in X$.
This yields the sequence $(M(F^n({ x })))_{n\geq 1}$ of matrices, called the \emph{continued fraction expansion} of $x$.

We then can interpret these matrices as incidence matrices of substitutions (with a choice that is highly non-canonical).
For explicit examples, see e.g.~\cite{BertheLabbe,delecroixwords} or else~\cite{Integers} and the references therein.
\begin{exa}[Jacobi-Perron substitutions] \label{exa:JP}
Consider for instance the Jacobi-Perron algorithm.
Its projective version is defined on the unit
square
$(0,1)\times (0,1)$ by:
$$(\alpha,\beta) \mapsto
\left(\frac{\beta}{\alpha}-\left\lfloor \frac{\beta}{\alpha}\right\rfloor,
\frac{1}{\alpha}-\left\lfloor\frac{1}{\alpha}\right\rfloor\right) =\left( \left\{\frac{\beta}{\alpha}\right\}, \left\{\frac{1}{\alpha}\right\}\right).$$
Its linear version is defined on the positive
cone
$X=\{(a,b,c)\in\mathbb{R}^3|0 < a,b<c\}$ by:
$$(a,b,c) \mapsto (a_1,b_1,c_1)=(b-\lfloor b/a\rfloor a,c-\lfloor c/a\rfloor a,a).$$

Set $B=\lfloor b/a\rfloor a$, $C=\lfloor c/a\rfloor .$
One has $$\left(\begin{array}{cc}
a\\
b\\
c
\end{array}\right)= \left(\begin{array}{lll}
0 & 0 & 1\\
1& 0 & B\\
0 & 1 & C
\end{array}\right)
\left(\begin{array}{cc}
a_1\\
b_1\\
c_1
\end{array}\right).$$
We associate with the above matrix the substitution $\sigma_{B,C} \colon 1 \mapsto 2, \ 2 \mapsto 3, \ 3 \mapsto 12^B 3^C$.
These substitutions have been studied in a higher-dimensional framework in \cite{ABI} for instance.
\end{exa} 

\begin{exa} \label{exa:3iet_simplex}
Here, we see that Example~\ref{exa:3iet} is intimately related to some continued fractions \cite{ra0}. Namely, one associates with each vertex $v$ of the graph in Example~\ref{exa:3iet} the cone $\Delta_v = \RR_+^3$. On each edge of the graph, consider the incidence matrix of the substitution. Then, one can notice that it defines a continued fraction algorithm: to a.e. vector $x \in \Delta_v$, there is only one possible inverse of the incidence matrix that can be applied in order to keep a positive vector. More precisely, a simple computation shows that each of the three simplices is cut into two parts and each of these parts is mapped to the entire simplex. This continued fraction algorithm corresponds to the Rauzy induction and encodes a cutting and stacking procedure for 3-interval exchange transformations.
\end{exa}

\subsubsection*{Lagrange Theorem and substitutive sequences}
Let us end this section with a further example of a problem for which the continued fraction formalism is particularly relevant. Given an $S$-adic word, one can ask whether this word is substitutive, or even purely substitutive. Substitutive Sturmian sequences correspond to quadratic angles, and purely substitutive ones are characterized in terms of the Galois Theorem for continued fraction expansions (for more details, see \cite{Ei} and the references therein). In the same vein, if all the parameters of an interval exchange belong to the same quadratic extension, the sequence of induced interval exchanges (by performing always the same induction process) is ultimately periodic \cite{Bosh} (the converse does not hold: there are example of substitutive $4$-iet with eigenvalues of degree $4$); see also \cite{Adam2,Jullian-iet}. More generally, recall the characterization given in \cite{Dur98} of primitive substitutive words based on the notion of return word (see Example \ref{ex:return}):
a uniformly recurrent word is substitutive if and only if the set of its derived words is finite. 
For most continued fraction algorithms, there is no algebraic description of the set of periodic points.

\section{Factor complexity}\label{sec:complexity}
Sturmian words (Example~\ref{exa:sturm}), Arnoux-Rauzy words (Example~\ref{exa:AR})
or interval exchange words (Example~\ref{exa:3iet}) can all be described as families of $S$-adic
words parametrized by infinite paths in an $S$-adic graph. Furthermore, each word in these families
has linear factor  complexity (more precisely,  complexity $n+1$ for Sturmian words, and complexity
$kn+1$ for Arnoux-Rauzy words or $k-1$-iet words).
Hence, we may expect that having a  `nice' $S$-adic description is somewhat related to  a  `nice' linear
complexity, as  illustrated in this section.  More generally, connections between $S$-adic words and linear factor complexity have been widely
investigated~\cite{frank,Dur03,DLR,Leroy,LR}. For other explicit computations and details about Rauzy
graphs see~\cite[Chap. 4]{CANT}. Let us insist on the recent article~\cite{DLR} which provides a
lot of detailed examples about this problem.

\subsection{Words of linear complexity are $S$-adic}
One direction of the correspondence between linear  factor complexity and $S$-repre\-sen\-tation is given by
the following result of S.~Ferenczi.
\begin{thm}[\cite{frank}] \label{thm:lin_comp_implies_adic}
  Let us fix $k$. Then, there exists a finite set of substitutions $S_k$ such that any uniformly recurrent word $u$
  whose complexity function $p_u(n)$ satisfies $p_u(n+1)-p_u(n) \leq k$, for all $n$, is $S_k$-adic.
\end{thm}
One  natural   framework of application  for     Theorem \ref{thm:lin_comp_implies_adic} comes from  the following deep combinatorial result: if an infinite word $u$ has at most linear factor complexity, then the first difference of the complexity function $p_u(n+1)-p_u(n)$ is bounded \cite{CA1}, and an  upper bound on the first difference can be explicitly given. This result was  thus a first and crucial step for producing \emph{finite}\footnote{According to Example \ref{ex:return}, $S$-adic expansions can be obtained thanks to return words, but the point is here to get expansions with finitely many substitutions.} sets $S $ of substitutions and $S$-adic representations for words with at most linear complexity. This program has   been achieved in \cite{frank} for uniformly recurrent infinite words with at most linear factor complexity function, thanks to 
Theorem \ref{thm:lin_comp_implies_adic}; see also the more recent version of \cite{Leroy,Leroy13,LR}.

Theorem~\ref{thm:lin_comp_implies_adic} has been refined by F. Durand and J. Leroy~\cite{DL,Leroy}, especially in the
case of very low complexity.
\begin{thm}[\cite{DL,Leroy}] \label{thm:2n+1}
Let $u$ be a recurrent  word whose complexity function $p_u$ satisfies $p_u(n+1) - p_u(n) \leq 2$ for all $n$.
Then, $u$ is $S$-adic on the following set of $5$ substitutions containing
\begin{enumerate}
 \item two Dehn twists $R\colon a \mapsto ab, \ b \mapsto b, \ c \mapsto c$ and $L\colon a \mapsto ba, \ b \mapsto b, \ c \mapsto c$,
 \item two permutations $E_{ab} \colon a \mapsto b, \ b \mapsto a, \ c \mapsto c$ and $E_{bc} \colon a \mapsto a, \ b \mapsto c, \ c \mapsto b$ (note that these permutations generate the symmetric group of $\{a,b,c\}$),
 \item a projection  $M_{ab} \colon a \mapsto a, \ b \mapsto b, \ c \mapsto b$.
\end{enumerate}
\end{thm}
For the proofs we refer to~\cite{DL,Leroy}. Note that only the Dehn twists and the permutations are needed to generate all
3-iet words (see Example~\ref{exa:3iet}) or Arnoux-Rauzy words (see Example~\ref{exa:AR}). Moreover these four
substitutions are exactly the ones which are invertible as a morphism of the free group.

An essential tool in the proof of Theorems~\ref{thm:lin_comp_implies_adic} and~\ref{thm:2n+1} are Rauzy
graphs and their evolutions (see for example~\cite[Chap. 4]{CANT}).

We warn the reader that not every directive sequence made of these substitutions generates an $S$-adic word of
complexity $2n+1$: Remark~\ref{rem:cass}  can be adapted to get words of exponential complexity.
A characterization of non-periodic minimal subshifts $X$ for which $p_X(n+1)- p_X(n) \leq 2$ for $n$ large enough can
be given, expressed for the directive sequences in terms of infinite paths in a labeled directed graph that satisfies
some conditions  which are not of finite type, see~\cite{DL,Leroy}.

\subsection{Alphabet growth and entropy}
Now, we start to investigate how  an $S$-adic representation can provide information on the factor  complexity.
In this section, we provide a bound on the entropy, i.e., the exponential growth of the complexity
function. Recall that we are mostly interested in linear  factor complexity, in particular, in  zero entropy.

We reproduce the following result due to T.~Monteil~\cite{Monteil} which gives a general bound on the entropy.
\begin{thm} \label{thm:zero_entropy}
 Let $(\sigma_n)_n$ be a sequence of substitutions, with $\sigma_n: \AAA_{n+1}^* \rightarrow \AAA_n^*$, and  let $X$ be  the associated $S$-adic shift.
 Let 
 \[
 \beta_n^- = \min_{i \in \AAA_n} |\sigma_{[0,n)}(i)|.
 \]
 Then,  the toplogical entropy $h_X$ of $X$ satisfies
 \[
 h_X \leq \inf_{n \geq 0} \frac{\log \Card \AAA_n}{\beta_n^-}.
 \]
 In particular, if $(\sigma_n)_n$ is everywhere growing and the alphabets $\AAA_n$ are of bounded cardinality, then $X$ has zero entropy.

 Conversely, if $u$ is an infinite word whose associated shift $X_u$ has entropy $h_{X_u}$, then it
 admits an everywhere growing $S$-adic representation $((\sigma_n,a_n))_n$, with $\sigma_n: \AAA_{n+1}^* \rightarrow \AAA_n^*$, such that
 \[
 h_{X_u} = \inf_{n \geq 0} \frac{\log \Card \AAA_n}{\beta_n^-}.
 \]
 Moreover, the substitutions can be chosen so that $\beta_n^- = \beta_n^+ = 2$ for all $n$.
\end{thm}
A weaker version of this result already appeared in~\cite{DLR}. As a consequence, to build an $S$-adic word with
positive entropy requires some growth conditions on the cardinalities of the  alphabets $\AAA_n$.

\begin{proof}
 Let $n$ be fixed. Let $W_n = \{\sigma_0 \ldots \sigma_{n-1}(i) \mid i \in \AAA_n\}$ and let 
 $
 \beta_n^+ = \max_{i \in \AAA_n} |\sigma_{[0,n)}(i)|.$
 By definition, any factor $w$ in $X$ can be decomposed as $w = p v_1 \ldots v_k s$ where the $v_j$ belong to $W_n$,
 $p$ is a suffix of an element of $W_n$ and $s$ a prefix.
 For any $N$ large enough, any factor $w$ of length $N$ is a factor of a concatenation of at most $\frac{N}{\beta_n^-} +2$
 words in $W_n$ (we include $p$ and $s$).
 By taking into account the possible prefixes, there are at most 
 $( \Card \AAA_n) ^{ \frac{N}{\beta_n^-} +2} \cdot (\beta_n^+) $
 words of length $N$, which gives 
 \[
 \frac{\log p_X(N)}{N} \leq \inf_{n \geq 0} \left(\left(\frac{1}{\beta_n^-} + \frac{2}{N} \right) \log \Card \AAA_n + \frac{ \log \beta_n^+}{N}\right).
 \]
 Now, we have $h_X = \inf \log p_X(N)/N$,  and letting $N$ tend to infinity in the above inequality yields the result.

 We now prove the second part of the theorem. Let $u$ be an infinite word whose associated shift $X_u$ has entropy $h_{X_u}$.
 Then, for each $n$, we cut the word $u$ into pieces of size $2^n$.
 More precisely, let
 \[
 V_n = \{u_{2^n} u_{2^n+1} \ldots u_{2^{n+1}-1} \mid n \in \NN\}.
 \]
 Note that $V_n$ is included in the set of factors of $u$.
 Let $d_n$ be the cardinality of $V_n$ and number arbitrarily
 the elements of $V_n$. As each word of length $2^{n+1}$ is the concatenation of two words of length $n$ we have naturally
 defined a substitution from $\AAA_{n+1}$ to $\AAA_n$ for which all images have length $2$.
 Let $d_n$ stand for the cardinality of  $\AAA_n$.
 By definition, $d_{n+1} \leq d_n d_n$ and hence $(\log d_n)_n$ is subadditive. Moreover, since  $V_n$ is included in the
 set of factors of $u$, we have
 \[
 \inf_n \frac{\log d_n}{2^n} = \lim_{n \to \infty} \frac{\log d_n}{2^n} \leq h_{X_u},
 \]
 and the first part of the proof provides the other inequality.
\end{proof}

\subsection{Two examples}
Theorem~\ref{thm:zero_entropy} shows that everywhere growing $S$-adic  representations with bounded
alphabets only provide words with zero entropy. Zero entropy does not distinguish between linear and
polynomial complexity. We reproduce here two examples issued  from~\cite{DLR} that show that
it might be subtle to provide a criterion for sequences of substitutions to  yield $S$-adic words
with linear complexity. Observe that a restriction on $S$-adic representations yielding to linear complexity cannot be
formulated uniquely in terms of the set $S$ of substitutions: there exist sets of substitutions
 which produce infinite words that have at most linear complexity function, or not, depending on the directive sequences (see Example \ref{exa:ler}).
 The $S$-adic conjecture (which rather should be qualified as
 problem)
thus consists in providing a characterization of the class of $S$-adic expansions that generate only words with linear factor complexity by formulating a suitable set of conditions on the set $S$ of substitutions together with the associated directive sequences. 

Recall that there exist substitutions which generate words with quadratic complexity.
An example of such a  substitution is $\sigma$ defined on $\{a,b\}^*$ by
\begin{equation} \label{eq:quad_comp}
\sigma: \left\{ \begin{array}{lll}
  a & \mapsto & aab \\ b & \mapsto & b
\end{array} \right. .
\end{equation}
It is not hard to  see that its language $L_\sigma$ contains $\{ab^ka \mid k \in \NN\}$
from which a lower  quadratic bound  for the  factor complexity can  easily be deduced
(see~\cite[Chap. 4]{CANT} for a precise computation).
Note that we have $|\sigma^n(b)| = 1$ for all $n$, namely the substitution is not
everywhere growing.

Recall also that in Remark~\ref{rem:cass}
which provides an $S$-adic expansion for any word, whatever its complexity function is , there is a
special letter $\ell$ for which $|\sigma_0 \sigma_1 \ldots \sigma_{n-1}(\ell) |= 1$ for all $n$.
So, one first idea in order to bound the complexity is to impose some growth on the lengths of the letters.

\begin{exa}\label{exa:ler}
Here we recall~\cite[Example 1]{Leroy} and refer to~\cite{Leroy} and~\cite{DLR} for a proof.
Let $S=\{\sigma, \tau\} $ with $\sigma$ as in~\eqref{eq:quad_comp} and $\tau\colon a \mapsto ab$, $b \mapsto ba$. Note that $\tau$ is the Thue-Morse substitution which is positive (hence the two  infinite words that it generates have linear complexity). Let $(k_n)_n$
be a sequence of non-negative integers, and let $ u$ be the $S$-adic word
\[
u = \lim_{ n \rightarrow \infty} \sigma ^{k_0}\, \tau\, \sigma^{k_1}\, \tau\, \cdots\, \tau\, \sigma^{k_n} (a).
\]
Then, the $S$-adic word $u$ has linear factor complexity if and only if the sequence $(k_n)_n$ is bounded.
\end{exa}

It is well known that a  product of primitive matrices is not necessarily primitive (this is why positivity
is a much better  property). An example of an  $S$-adic expansion which only involves primitive substitutions, but
for which the limit does not have linear complexity is recalled below  in the spirit
of~Example \ref{exa:ler} above.

\begin{exa}\label{exa:prim}
The following example is due to J. Cassaigne and N. Pytheas Fogg. For a proof, see~\cite{DLR}.
Let $S=\{\sigma, \mu \} $ with $\sigma$ as in~\eqref{eq:quad_comp} and
$\mu \colon a \mapsto b$, $ b \mapsto a.$
Let
\[
v = \lim_{ n \rightarrow \infty } \sigma \, \mu \, \sigma^2 \, \mu \, \sigma^3 \, \mu \, \sigma^4
\, \cdots \, \mu \sigma^{n} \, \mu (b).
\]
The word $v$ is uniformly recurrent and it does not have linear factor complexity.
As   already noticed, $\sigma$ is not primitive. But the $S$-adic representation above may be rewritten in terms
of another set of substitutions $S'=\{\sigma\, \mu, \mu \, \sigma\}$
(by noticing that $\mu ^2=\mbox{Id}$). Observe that the substitutions $\sigma \, \mu$ and $\mu \, \sigma$ are primitive
\end{exa}

\subsection{Global complexity}
Let us finish this section by a remark on global complexity.

Given an $S$-adic graph $(G,\tau)$ that provides a parametrization of a family of $S$-adic words (such as the set of Sturmian words, or Arnoux-Rauzy words), different languages are of interest: the \emph{local languages}, with each of them being generated by a given $S$-adic word $u$ associated with a given directive sequence 
$\gamma \in \Sigma_G$, and the \emph{global language} at a vertex $v$, defined as the union of the local languages for the elements $\gamma \in \Sigma_G$ such that $\gamma_0 = v$. We call \emph{global complexity} of a family the complexity function of its global language.

The global complexity of Sturmian words has been investigated in~\cite{Lipatov,Mignosi,BerstelPocchiola}, and for $3$-iet words in~\cite{AFMP}.
Those proofs do not use the $S$-adic properties of these languages but rather arithmetics or  planar geometry.
In~\cite{CassaignePetrovFrid}, such a  computation is performed for  an example related to Toeplitz words, which 
 uses the $S$-adic structure (that transfers as equations on the  factor complexity function). All these
examples show that global complexity tends to be polynomial.

We say that an $S$-adic graph is \emph{positive} if, for any edge in this graph, the associated sequence of substitutions is positive. 
According to \cite{Dur03} (see also Theorem~\ref{thm:primitivity_implies_ue_adic}), if $(G,\tau)$ is a finite positive $S$-adic graph, then there exists a constant $C$ such that for any sequence $\gamma \in \Sigma_G$, the factor complexity of the local language associated with $\gamma$ satisfies $p_\gamma(n) \leq Cn$, for all $n$. In~\cite{delecroix-comp}, it is proven that the global language for $(G,\tau)$ has a polynomial behavior. Moreover, the exponent of the polynomial is related to Lyapunov exponents that will be discussed in
Section~\ref{sec:Lyapunov_exponents}.

\section{Minimality, frequencies and invariant measures} \label{sec:dynamics}
Recall that if $\sigma$ is a primitive substitution, then the dynamical system $(X_{\sigma},T)$ is
linearly recurrent: this is a strong form of minimality which implies unique ergodicity
(see Theorem~\ref{thm:subs_syst_are_linearly_recurrent} and Theorem~\ref{thm:LR}). Now, we
investigate these properties for $S$-adic systems.

\subsection{Minimality and linear recurrence}
In this section we introduce two notions of primitivity for $S$-adic expansions. We relate them respectively to
minimality (Theorem~\ref{thm:minimality}) and linear recurrence (Theorem~\ref{thm:primitivity_implies_ue_adic}
and~\ref{thm:linearly_recurrent}).

\begin{dfn}[Primitivity]
An $S$-adic expansion with directive sequence $(\sigma_n)_{n \in {\mathbb N}}$ is said \emph{weakly primitive} if, for each $n$, there exists $r$ such that the substitution $\sigma_{n} \, \cdots \, \sigma_{{n+r}}$ is positive.

It is said {\em strongly primitive} if the set of substitutions $\{\sigma_n\}$ is finite, and if there exists $r$ such that the substitution $\sigma_{n} \, \cdots \, \sigma_{{n+r}}$ is positive, for each $n$.
\end{dfn}
Note that primitivity expressed in matricial terms is called non-stationary primitivity in~\cite{Fisher}, and strong primitivity is called primitivity in~\cite{Dur03}.

\begin{thm} \label{thm:minimality}
An infinite word $u$ is uniformly recurrent (or the shift $X_u$ is minimal) if and only
if it admits a weakly primitive $S$-adic representation.
\end{thm}
\begin{proof}
We first prove that an $S$-adic word with weakly primitive expansion is minimal. Let $ (\sigma_n)_n$ be the weakly primitive directive sequence of substitutions of an $S$-adic representation of $u$. Observe that this representation is necessarily everywhere growing. Consider a factor $w$ of the language. It occurs in some $\sigma_{[0,n)}(i)$ for some integer $n \geq 0$ and some letter $i \in \AAA$. By definition of weak primitivity, there exists an integer $r$ such that $\sigma_{[n,n+r)}$ is positive. Hence $w$ appears in all images of letters by $\sigma_{[0,n+r)}$ which implies uniform recurrence.

Conversely, let $u$ be a uniformly recurrent sequence on $\AAA = \AAA_0$. Recall that a return word of a factor $w$ is a word separating two successive occurrences of the factor $w$ in $u$. According to Example~\ref{ex:return}, let us code the initial word $u$ with these return words; one obtains an infinite word $u'$ on a finite alphabet that is still uniformly recurrent ($u'$ is a derived word). By repeating the construction, one obtains an $S$-adic representation of $u$. That $S$-adic expansion is weakly primitive, and we refer to~\cite{DHS99,Dur03,Leroy} for the details.
\end{proof}

The following statement illustrates the fact that strong primitivity plays the role of primitivity in the $S$-adic context (compare with Theorem~\ref{thm:subs_syst_are_linearly_recurrent}); see also \cite{AC} in the same vein.
\begin{thm}[\cite{Dur03}] \label{thm:primitivity_implies_ue_adic}
Let $S$ be a finite set of substitution and $u$ be an $S$-adic word having a strongly primitive $S$-adic expansion. Then, the associated shift $(X_u,T)$ is minimal (that is, $u$  is uniformly recurrent), uniquely ergodic, and it has at most linear factor complexity.
\end{thm}

\begin{rem}
In particular (see also \cite{DLR}), if $S$ is a set of non-erasing substitutions and $\tau \in S$ is positive,
the infinite word generated by a directive sequence for which $\tau$ occurs with bounded gaps is uniformly recurrent and has at most linear factor complexity.
See~\cite{Dur03} and~\cite{Leroy} for similar statements.\end{rem}

Strong primitivity is thus closely related to linear recurrence, that can be considered as a property lying in between being substitutive and being $S$-adic. With an extra condition of properness one even obtains the following characterization of linear recurrence. A substitution over $\AAA$ is said \emph{proper} if there exist two letters $b,e \in {\mathcal A}$ such that for all $a\in \AAA$, $\sigma(a)$ begins with $b$ and ends with $e$. An $S$-adic system is said to be proper if the substitutions in $S$ are proper.
\begin{thm}[Linear recurrence \cite{Dur03}] \label{thm:linearly_recurrent}
A subshift $(X,T)$ is linearly recurrent if and only if it is a strongly primitive and proper $S$-adic subshift.
\end{thm}
Let us mention that an essential ingredient in the proofs of Theorem~\ref{thm:primitivity_implies_ue_adic} and
Theorem~\ref{thm:linearly_recurrent} is the uniform growth of the matrices $M_{(0,n)}$ as it was the case for
substitutive systems (see~\eqref{eq:growth_matrix} in the proof of Theorem~\ref{thm:subs_syst_are_linearly_recurrent}).
We refer to~\cite{Dur03} for the proof.
 
With the following example, we stress the fact that strong primitivity alone
does not imply linear recurrence. Indeed, linear recurrence requires the property of being also proper $S$-adic.
\begin{exa}
 We recall the example of~\cite{Dur03} of a strongly primitive $S$-adic word, that is both uniformly recurrent, that has linear factor complexity, but that is not linearly recurrent. Take $S=\{\sigma, \tau\}$ with $\sigma \colon a \mapsto acb, \ b \mapsto bab, \ c \mapsto cbc$,
$\tau \colon a \mapsto abc, \ b \mapsto acb, \ c \mapsto aac$, and consider the $S$-adic expansion
\[
\lim_{n\rightarrow +\infty} \sigma \, \tau\, \sigma \, \tau \, \cdots \, \sigma^n \tau (a).
\]
\end{exa}



\subsection{Invariant measures}\label{subsec:mes}
This section is devoted to frequencies of $S$-adic systems, and more generally, to the invariant measures. 
Recall from Remark~\ref{rem:Bosh}  that factor complexity might be used to
provide a bound on the number of ergodic invariant measures.

As stated in Theorem~\ref{thm:cone_and_frequencies}, given a directive sequence $(\sigma_n)_n$, the limit cone determined by the incidence matrices of the substitutions $\sigma_n$, namely $ \cap_n M_{[0,n)} \RR^d_+$, is intimately related to letter frequencies in the corresponding $S$-adic shift. More generally, $S$-adic representations prove to be convenient to find invariant measures. Nevertheless, the situation is more contrasted for $S$-adic systems than for substitutive dynamical systems, for which primitivity implies unique ergodicity. This is well known since Keane's counterexample for unique ergodicity for 4-interval exchanges \cite{Keane}: weak primitivity does not imply unique ergodicity. For references on the widely investigated relations between primitivity and unique ergodicity, see the bibliography in~\cite{Fisher}.

Recall that for a primitive matrix $M$, the cones $M^n \mathbb{R}_+ ^d $ nest down to a single line directed by this eigenvector at an exponential convergence speed, according to the Perron-Frobenius Theorem (see e.g. \cite{Seneta}). The following condition is a sufficient condition for the sequence of cones $M_{[0,n)} \RR_+^d$ to nest down to a single strictly positive direction as $n$ tends to infinity (provided that the square matrices $M_n$ have all non-negative entries); in other words, the columns of the product $M_{[0,n)}$ tend to be proportional.
\begin{thm}[{\cite[pp.~91--95]{Furstenberg:60}}]\label{thm:furs}
 Let $(M_n)_n$ be a sequence of non-negative integer matrices and note $M_{[0,n)} = M_0 M_1 \ldots M_{n-1}$. Assume that there exists a strictly positive matrix~$B$ and indices $j_1 < k_1 \le j_2 < k_2 \le \cdots$ such that $B = M_{j_1} \cdots M_{k_1-1}= M_{j_2} \cdots M_{k_2-1} = \cdots$. Then,
\[
\bigcap_{n\in\mathbb{N}} M_{[0,n)} \RR^d_+ = \RR_+ f \quad \mbox{for some positive vector $f \in \RR_+^d$.}
\]
\end{thm}

The proof of that theorem relies on classical methods for non-negative matrices, namely Birkhoff contraction coefficient estimates and projective Hilbert metric~\cite{Birk}.

The vector $f$ of Theorem \ref{thm:furs}, when normalized so that the sum of its coordinates equal $1$, is called the \emph{generalized right eigenvector} associated with the $S$-adic representation. Nevertheless, there is no way to define a left eigenvector as the sequence of rows vary dramatically in the sequence $M_{[0,n)}$.


Now, we consider frequencies of words or, equivalently, invariant measures.
\begin{thm} \label{thm:invariant_measures}
Let $X$ be an $S$-adic shift with directive sequence $\tau = (\tau_n)_{n}$ where $\tau_n \colon \AAA_{n+1}^* \rightarrow \AAA_n^*$ and $\AAA_0 = \{1,\ldots,d\}$. Denote by $(M_n)_n$ the associated sequence of incidence matrices.
We assume that the sequence $(\tau_n)_n$ is everywhere growing.

Then, the cone
\[
C^{(0)} = \bigcap_{n \to \infty} M_{[0,n)} \RR_+^d
\]
parametrizes the letter frequencies: the set of vectors $f \in C^{(0)}$ such that $f_1 + \ldots + f_d = 1$ coincides
with the image of the map which sends    a  shift-invariant probability measure $\mu$ on $X$ to  the vector of
letter frequencies $(\mu([1]),\mu([2]),\ldots,\mu([d]))$. In particular $X$ has uniform letter frequencies if and only if
the cone $C^{(0)}$ is one-dimensional.

Furthermore,  the $S$-adic dynamical system $(X,T)$ is uniquely ergodic if and only if, for each $k$,  the limit cone
\[
C^{(k)}= \bigcap_{n \to \infty} M_{[k,n)} \RR^d_+
\]
is one-dimensional.
\end{thm}
In \cite{Bezugly3}, a somewhat finer version of Theorem~\ref{thm:invariant_measures} is proved in the context of Bratteli diagrams where a similar limit cone is identified to the set of invariant ergodic probability measures. Nevertheless, this  statement   does not apply directly to general $S$-adic systems such as underlined in~\cite{FFT}. For more on the connections between Vershik adic systems and $S$-adic ones, see \cite{Bezugly}, \cite[Sec. 5.2]{Bezugly2}, \cite[Chap. 6]{CANT}, and \cite{FFT}. Lastly, see~\cite{FFT} for further examples of minimal but not uniquely ergodic constructions of adic systems,  and for the reference therein.


\begin{rem} \label{rem:no_bounded_adic}
As observed in~\cite{DLR}, if an $S$-adic expansion is everywhere growing, then the number of ergodic
measures is bounded by $\max_n \Card \AAA_n$. Hence, there exist infinite words of zero entropy which
do not admit an everywhere growing $S$-adic representation with $S$ finite: take a subshift of zero
entropy with infinitely many ergodic invariant measures. For more
details, see the discussion in~\cite[Section 4]{DLR}.
\end{rem}

\begin{proof}[Proof of Theorem~\ref{thm:invariant_measures}]
The proof we adopt here differs somehow from the standard proof for proving uniform frequencies of shift generated by a primitive substitution \cite{Queffelec:10}.
Let $L^{(k)}$ stand for $S$-adic shift $X$ associated with $(\sigma_n)_{n \geq k}$, that is, $
L^{(k)} = \bigcap \overline{ \tau_{[k,n)} (\AAA_n^*)}.
$
We have $L(X) = L^{(0)}$.

Recall that a shift space $X$ has uniform letter frequencies if, for every letter $i$ in $\AAA_0$, the number of occurrences of $i$ in any factor $w$ of $L(X)$ divided by $|w|$ has a limit when $|w|$ tends to infinity, uniformly in the length of $w$. Unique ergodicity is equivalent to uniform factor frequency (just replace `letter' by `factor' in the previous sentence).

Let $n$ be a positive integer and let $W_n = \{\tau_{[0,n)}(a) \mid a \in \AAA\}$ be the set of images of letters by $\tau_{[0,n)} = \tau_0 \ldots \tau_{n-1}$. We note
\[
\beta^{-}_n = \min_{v \in W_n} |v|,
\quad
\beta^{+}_n = \max_{v \in W_n} |v|
\quad \text{and} \quad
\delta_n = \max_{a \in \AAA_0, w \in W_n} \left| \frac{|w|_i}{|w|} - f^{(0)}_a \right|.
\]
By hypothesis $\beta^{-}_n$ and $\beta^{+}_n$ tend to infinity, while $\delta_n$ tends to zero.
\medskip

{\bf Uniform letter frequencies.}
We first prove that letters have uniform frequencies in $L ^{(0)}$ under the condition that the cone $C^{(0)}$ is one-dimensional.
Let us fix $n$ ($n$ will later tend again to infinity), and 
let $\tau = \tau_{[0,n)}.$

According to the Dumont-Thomas prefix-suffix decomposition \cite{DT89}, any word $w$ in $L ^{(0)}$ may be decomposed as $w = s \tau(w') p$, 
where $w'=(w'_i)_i$ and $\tau(w'_i)$ belongs to $W_n$, for all $i$, $s$ is a suffix of an element of $W_n$, and $p$ is a prefix of an element of $W_n$. We obtain for any $w$ in $L^{(0)}$ and any letter $a$ in $\AAA_0$
$$ \left| \frac{|w|_a}{|w|} - f^{(0)}_a \right| \leq \left| \frac{|w|_a}{|w|} - \frac{|\tau(w')|_a}{|w|} \right|+ \left| \frac{|\tau(w')|_a}{|w|}- \frac{|\tau(w')|_a}{|\tau(w')|} \right|+
 \left| \frac{|\tau (w')|_a}{|\tau(w')|}-f^{(0)}_a \right| ,$$
 which gives 
$$\left| \frac{|w|_a}{|w|} - f^{(0)}_a \right| \leq \frac{4 \beta^{+}_n}{|w|} + \delta_n.$$

Hence, for all $w$ of length larger than $4 \beta^{+}_n / \delta_n$, the error of frequency in $w$ is at most $2 \delta_n$.
As $\delta_n$ tends to $0$, we have uniform frequencies in $L ^{(0)}$.

\medskip

\textbf{Uniform frequencies for words.} Now we prove that $L^{(0)}$ has uniform frequencies for words under the condition that all the cones $C^{(k)}$ are one-dimensional. Using the first part, all languages $L^{(k)}$ have also uniform letter frequencies, with the corresponding frequency vector being provided by $f^{(k)}$ defined as the unique non-negative vector in $C^{(k)}$ such that $f^{(k)}_1 + \ldots + f^{(k)}_{d_k} = 1$ ($d_k =\Card \AAA_k$).

Let $v$ be a finite word in $L^{(0)}$. Fix again $n$ large (which will later tend to infinity), and set
$\tau = \tau_0 \tau \cdots \tau_{n-1}$. Let 
\[
g_{v,n} := \frac{\sum_{a \in \AAA_n} f^{(n)}_a |\tau(a)|_v}{\sum_{a \in \AAA_n} f^{(n)}_a |\tau(a)|} \in [0,1].
\]
We first claim that
\begin{equation}\label{eq:freq}
 \lim_{w' \in L^{(n)}; \, |w'| \to \infty} \frac{\sum_{m=0}^{|w'|-1} |\tau(w'_m)|_v}{|\tau(w')|} = g_{v,n}.
\end{equation}
To do so, we apply the fact that $L^{(n)}$ has uniform letter frequencies. 
It implies that Birkhoff sums of functions that only depend on the first letter
converge to their mean (with respect to the frequency).
Applying this result to the two functions $a \mapsto |\tau(a)|$ and $a \mapsto |\tau(a)|_v$, we get
\[
\lim_{|w'| \to \infty} \frac{|\tau(w')|}{|w'|} = \sum_{a \in {\AAA}_n} f_a^{(n)} |\tau(a)|
\quad \text{and} \quad
\lim_{|w'| \to \infty} \frac{\sum_{m=0}^{|w'|-1} |\tau(w'_m)|_v}{|w'|} = \sum_{a \in {\AAA}_n} f_a^{(n)} |\tau(a)|_v.
\]
Hence the claim follows.

Now, we proceed as in the first part. Any word $w$ in $L^{(0)}$ can be decomposed as $w = s \tau(w') p$,
where $w'=w'_0\ldots w'_k$ belongs to $L^{(n)}$, and the sizes of both $s$ and $p$ are smaller than or equal to $\beta^{+}_n$.
Each occurrence of $v$ in $\tau(w')$ either occurs inside the image of a letter $\tau(w'_m)$, or overlaps on the image of two consecutive ones $\tau(w'_m) \tau(w'_{m+1})$. The latter case occurs at most $|v|(|w'|-1)$ times. Since $|w|/ \beta^{+}_n - 2 \leq |w'| \leq |w|/\beta^{-}_n$, we obtain the upper bound
\[
\left| \frac{|w|_{v}}{|w|} - g_{v,n} \right| \leq
\frac{2 \beta^{+}_n}{|w|} +
\frac{|v|}{\beta^{-}_n} +
\left|\frac{\sum_{m=0}^{|w'|-1}|\tau(w'_m)|_v}{|\tau(w')| + |p| + |s|} - g_{v,n}\right|
\]
where the three terms on the right hand side bound respectively the occurrences of $v$ in $p$ and $s$, the occurrences of $v$ occurring at the concatenation of $\tau(w'_m)\tau(w'_{m+1})$, and lastly, the occurrences of $v$ inside the $\tau(w'_m)$.
By making $|w|$ tend to infinity, we obtain
\[
\limsup_{w \in L^{(0)},\  |w| \to \infty} \left| \frac{|w|_v}{|w|} - g_{v,n} \right| \leq \frac{|v|}{\beta_n^-}.
\]
Note that we use the claim~\eqref{eq:freq} for the third term.
This shows that $(g_{v,n})_n$ is a Cauchy sequence and that $v$ has uniform frequencies in $L^{(0)}$.
 
\smallskip

Now, assume that the cone $C^{(0)}$ is not one-dimensional and consider an extremal point $f$ in it.
Then, since extremal points are mapped to extremal points, there exists a sequence $(i_n)_n$ of letters, with $i_n \in \AAA_n$ for all $n$, such that
\[
\RR_+ f = \lim_{n \to \infty} M_{[0,n)} e_{i_n}
\]
where $(e_1, e_2, \cdots,e_{d_n})$ stands for the canonical basis in $\RR^{d_n}$ ($d_n =\Card \AAA_n$). Now, take a limit point $u$ of the sequence of words $(\tau_{[0,n)}(i_n))_n$. It is not hard to see that $f$ belongs to the set of limit points 
of the sequence of vectors $(|u_0 \ldots u_n|_1, |u_0 \ldots u_n|_2, \ldots, |u_0 \ldots u_n|_d)$.
Hence $f$ is the frequency vector of a sequence in $X$.
\end{proof}

\subsection{Balancedness}
Recall that an infinite word $u \in \AAA^\NN$ is said to be \emph{balanced} if there exists a constant $C>0$ such that for any pair $v,w$ of factors of $u$ of the same length, and for any letter $i \in \AAA$, $||v|_i - |w|_i| \leq C$. We let $B(u)$ denote the smallest constant such that $u$ is $B(u)$-balanced.

Let $u \in \AAA^\NN$ be an infinite word and assume that each letter $i$ has frequency $f_i$ in $u$. The \emph{discrepancy} of $u$ is
\[
\Delta(u) = \limsup_{i \in \AAA, \ n \in \mathbb{N}} | | u_0 u_1 \ldots u_{n-1}| _i - n f_i|.
\]
These quantities $B(u)$ and $\Delta(u)$ are considered e.g. in \cite{Adam0,FBRS,tij} in connection with bounded remainder sets.

According to Proposition \ref{prop:balanced}, 
balancedness implies the existence of uniform letter frequencies. Moreover, if $u$ has letter frequencies, then $u$ is balanced if and only if its discrepancy $\Delta(u)$ is finite. 
 We always have $\Delta(u) \leq B(u)$, and we have $B(u) \leq 4 \Delta(u)$ if they are both finite (see the proof of Proposition \ref{prop:balanced} and \cite[Proposition 7]{Adam0}).

Sturmian words are known to be $1$-balanced \cite{Lothaire:2002}; they even are exactly the $1$-balanced infinite words that are not eventually periodic. There exist Arnoux-Rauzy words that are not balanced such as first proved in~\cite{CFZ00}, contradicting the belief that they would be natural codings of toral translations. For more on this subject, see also \cite{BFZ,CFM08,delecroixwords}. In~\cite{BCS} sufficient conditions that guarantee balancedness are provided for Arnoux-Rauzy words. In particular, bounded partial quotients in the $S$-adic expansion imply balancedness, but there exist $2$-balanced Arnoux-Rauzy words with unbounded partial quotients.

Let $u = (u_n)_n \in \AAA^\NN$ be an infinite word and $\phi: \AAA^\NN \rightarrow \RR$ be a continuous function. We will often abuse notation and if $\phi$ is a function from $\AAA$ to $\RR$, we note also $\phi$ the (continuous) function from $\AAA^\NN$ to $\RR$ defined by $x \mapsto \phi(x_0)$.
 The \emph{Birkhoff sum} of $\phi$ along $u$ is the sequence
\[
S_n(\phi,u) = \phi(u_0) + \phi(u_1) + \ldots + \phi(u_{n-1}).
\]

It generalizes the concept of frequency: indeed, if $\phi= \chi_{[i]}$ is the characteristic function of the letter $i$, then $S_n(\phi,u)$ is the number of occurrences of $i$ in the prefix of length $n$ of $u$. More generally, if $(X_u,T)$ is the symbolic dynamical system generated by $u$, and $\phi: A^\NN \rightarrow \RR$ is a continuous function, we may define
\[
S_n(\phi,T,x) = \phi(x) + \phi(Tx) + \ldots + \phi(T^{n-1}x),
\]
for all $x\in X_u$. In the context of symbolic dynamics, this corresponds to take the sum of the values of $\phi$ along an orbit. We use the notation $S_n(\phi,T,x)$ in order to emphasize the role of the shift $T$.

Assume that $(X_u,T)$ has uniform word frequencies and let $f$ denote the letter frequencies vector. Then, uniformly in $x \in X_u$, we have, by unique ergodicity,
\[
\lim_{n \to \infty} \frac{S_n(\phi,T,x)}{n} = \sum_{i \in A} \phi(i) f_i.
\]
Now, $u$ is balanced if and only if there exists a constant $C$ so that
\[
\left| \frac{S_n(\phi,T,x)}{n} - \sum_{i \in A} \phi(i) f_i \right| \leq \frac{C \|\phi\|}{n}, \quad \text{for all $n \geq 0$.}
\]
In other words, balancedness may be interpreted as an optimal speed of convergence of Birkhoff sums. In Section~\ref{sec:Lyapunov_exponents} we study speed of convergence of $S$-adic systems in relation with Lyapunov exponents. 

We now prove a criterion for balancedness (or bounded discrepancy) of $S$-adic sequences.
If $M$ is a $d \times d$ matrix and $E$ is a vector subspace of $\RR^d$ we define
\[
\|M|_E\| := \sup_{v \in E \backslash \{0\}} \frac{\|M v\|}{\|v\|}.
\]
\begin{thm} \label{thm:ness_criterion}
Let $(\sigma_n)_n$ be the directive sequence of an $S$-adic representation of a word $u$. For each $n$, let $M_n$ be the incidence matrix of $\sigma_n$. Assume that $u$ has uniform letter frequencies and let $f$ be the letter frequencies vector. If
\[
\sum_{n \geq 0} \| ( {}^t M_{[0,n)}) |_{f^\perp} \|\ \| M_n \| < \infty,
\]
then the word $u$ is balanced.
\end{thm}
Note that if the substitutions $\sigma_n$ belong to a finite set, then the norms $\|M_{n}\|$ are uniformly bounded and can be removed from the sum.
For an example of application of this theorem to Brun substitutions, see \cite{delecroixwords}.
\begin{proof}
The proof is similar to~\cite{Adam4} and uses the Dumont-Thomas prefix-suffix decomposition \cite{DT89} of the factors of $u$.
For any factor $w$, there exists an integer $n$ and words $p_j,s_j \in \AAA_j^*$, for $j=0,\ldots,n$, 
 such that
\[
w = p_0 \sigma_0(p_1 \ldots \sigma_{n-1}(p_{n-1} \sigma_n(p_{n+1} s_{n+1}) s_{n-1}) \ldots s_1) s_0
\]
where each $p_j$ is a suffix of an image of a letter by $\sigma_{j}$, and each $s_i$ is a prefix of an image of a letter by $\sigma_j$.
This decomposition can easily be proved by induction.

Let $\sigma_{[0,n)} = \sigma_0 \sigma_1 \ldots \sigma_{n-1}$ and let $i$ be a letter in $\AAA_0$. One has 
\[
|w|_i - |w| f_i =
\sum_{k=0}^{n+1} \left(|\sigma_{[0,k)}(p_k)|_i - |\sigma_{[0,k)} (p_k)| f_i + |\sigma_{[0,k)}(s_k)|_i - |\sigma_{[0,k)}(s_k)| f_i \right).
\]
Let $(e_1,\ldots,e_d)$ stand for the canonical basis of $\mathbb{R}^d$, where $d$ is the cardinality of the alphabet $\AAA_0$.
Now, the vector $e_i - f_i (e_1+ \cdots + e_d)$ belongs to the orthogonal of $f$ (since $f_1+\cdots+f_d=1$). Let $u_k$ and $v_k$ be respectively the abelianizations of $p_k$ and $s_k$, i.e., the vectors in $\RR^d$ that count the number of occurrences of each letter. Then, we may rewrite the previous sum as
\[
|w|_i - |w| f_i =
\sum_{k=0}^{n+1} \langle (e_i - f_i (e_1+ \cdots + e_d)), M_{[0,k)} (u_k + v_k) \rangle,
\]
from which we obtain
\begin{align*}
||w|_i - |w| f_i|
& \leq \sum_{k=0}^{n+1} \left| \langle (e_i - f_i (e_1+ \cdots + e_d)), M_{[0,k)} (u_k + v_k) \rangle \right| \\
& = \sum_{k=0}^{n+1} \left| \langle {}^t M_{[0,k)} (e_i - f_i (e_1+ \cdots + e_d)), (u_k + v_k) \rangle \right| \\
& \leq 2 \sum_{k=0}^{n+1} \| ({}^t M_{[0,k)})|_{f^\perp} \| \|M_{k}\|,
\end{align*}
which ends the proof by Proposition \ref{prop:balanced}.
\end{proof}

\section{Diophantine exponent of approximation and a quantitative Birkhoff Theorem} \label{sec:Lyapunov_exponents}
According to Theorem~\ref{thm:invariant_measures}, the convergence of the cones $M_0 \ldots M_{n-1} \RR^d_+$ to a generalized eigendirection yields  uniform frequencies of the corresponding $S$-adic system. On the other hand, uniform frequencies correspond to the convergence of Birkhoff averages. In this section we study a quantitative version of these two related convergences.

In terms of continued fractions, the convergence of the nested cones is related to the quality of Diophantine approximation. It is measured by the ratio of the first two Lyapunov exponents (see Section~\ref{subsec:cf}). Moreover, the very same quantity determines the growth of the discrepancy in Birkhoff sums (see Section~\ref{subsec:deviation_Birkhoff_and_balancedness}). We will see in particular that the property of balancedness holds generically as soon as an $S$-adic generalization of the Pisot property for substitutions holds, expressed in terms of Lyapunov exponents.

The main tool in the present section is ergodic theory at the level of directive sequences. For that reason, most theorems apply only for generic sequences with respect to some measure. It is possible to get relations for individual $S$-adic words along the lines of Theorem~\ref{thm:ness_criterion}. This kind of methods has been developed for example in~\cite{AthreyaForni} to provide estimates of deviations of Birkhoff averages for billiards in rational polygons.

\subsection{Matrix cocycle over an $S$-adic graph and Lyapunov exponents}\label{subsec:definition_Lyapunov}
Let $(G,\tau)$ be an $S$-adic graph as defined in Section~\ref{sec:graph}. Recall that $\Sigma_G$ denotes the set of infinite paths in $G$, and that with any infinite path $\gamma$, we associate an $S$-adic shift $X_G(\gamma)$. So $\Sigma_G$ may be thought as a parametrization of the family of $S$-adic shifts $X_G(\gamma)$. Given an infinite path $\gamma = \gamma_0 \gamma_1 \ldots \in \Sigma_G$, we define
\[
A_n(\gamma) = M_{\gamma_0} M_{\gamma_1} \ldots M_{\gamma_{n-1}},
\]
where $M_e$ is the incidence matrix of the substitution $\tau_e$ on the edge $e$. In particular, $A_1 = M_{\gamma_0}$ and we have the so-called cocycle relation
\[
A_{m+n}(\gamma) = A_m(\gamma) A_n(T^m \gamma),
\]
where $T: \Sigma_G \rightarrow \Sigma_G$ is the shift on $\Sigma_G$. One can consider the sequence $A_n(\gamma)$ as a non-commutative version of a Birkhoff sum: matrices are multiplied along an orbit of the dynamical system $T$.

We consider a shift-invariant probability measure $\mu$ on the graph $\Sigma_G$. In case of ergodicity of $\mu$, many properties of $S$-adic words follow the Kolmogorov 0-1 law with respect to $\mu$. See in particular Theorem~\ref{thm:generic_ue_and_separation} for unique ergodicity.

The Lyapunov exponents of the cocycle $A_n$ with respect to an ergodic probability measure $\mu$ are the exponential growth of eigenvalues of the matrices $A_n$ along a $\mu$-generic path $\gamma$. Lyapunov exponents were first defined by Furstenberg~\cite{Furs,Furstenberg:60}, and, in a sense, their existence generalizes the Birkhoff ergodic Theorem in a non-commutative setting. For general references on Lyapunov exponents, we refer to~\cite{BougerolLacroix} and~\cite{Furman}. For the  sake of simplicity, we assume that the incidence matrices $M_e$ are invertible for all edges $e$ in $G$ (in other words $A_1(\gamma)$ belongs to $\GL(d,\RR)$). We say that the cocycle $A_n$ is \emph{$\log$-integrable} if
\[
\int_{\Sigma_G} \log \max (\|A_1(\gamma)\|,\|A_1(\gamma)^{-1}\|) d\mu(\gamma) < \infty.
\]
If the matrices $A_1(\gamma)$ are bounded (e.g. the $S$-adic graph is finite), then this condition is automatically satisfied. When the matrices are not invertible, or without $\log$-integrability, one may obtain infinite Lyapunov exponents.

Assuming the ergodicity of $\mu$ and the $\log$-integrability of $A_1$, the first Lyapunov exponent of $(G,\Sigma_G,\tau)$ is the $\mu$-a.e. limit
\[
\theta_1^{\mu} = \lim_{n \to \infty} \frac{\log \|A_n(\gamma)\|}{n}.
\]
The first Lyapunov exponent measures the exponential growths of the incidence matrices of the products $\tau_{\gamma_0} \ldots \tau_{\gamma_{n-1}}$ along a $\mu$-generic sequence $\gamma_0 \gamma_1 \cdots$. The other Lyapunov exponents $\theta^{\mu}_2 \geq \theta^{\mu}_3 \ldots \geq \theta^{\mu}_d$ may be defined recursively by the almost everywhere limits
\[
\forall k=1,\ldots,d, \quad \theta_1 ^{\mu}+ \theta_2^{\mu} + \ldots + \theta_k ^{\mu}= \lim_{n \to \infty} \frac{\log \|\wedge^k A_n(\gamma)\|}{n},
\]
where $\wedge^k$ stands for the $k$-th exterior product. We will mostly be interested by the two first Lyapunov exponents $\theta_1^\mu$ and $\theta_2^\mu$.

A useful characterization of $\theta_2^\mu$ is as follows. Assume that for a.e. $\gamma$ the sequence of nested cones $A_n(\gamma) \RR^d_+$ converges to a line $f(\gamma)$. Then, we have the $\mu$-almost everywhere limit
\begin{equation} \label{eq:theta2}
\theta^\mu_2 = \lim_{n \to \infty} \frac{\log \|A_n|_{f(\gamma)^\perp}\|}{n}
\end{equation}
where, as in Theorem~\ref{thm:ness_criterion}, $\displaystyle \|A_n(\gamma)|_{f(\gamma)^\perp}\| = \sup_{v \in f(\gamma)^\perp} \frac{\|A_n v\|}{\|v\|}$.

In the context of an  $S$-adic graph, it is easy to obtain the positivity and simplicity of $\theta^\mu_1$ from non-negativity of the matrices. For more general matrix cocycles one needs further assumptions and the proof is more involved (see e.g. \cite{BougerolLacroix}). Let us  say that a finite path $\gamma_0 \ldots \gamma_{k-1}$ in an $S$-adic graph is \emph{positive} if the associated matrix $M_{\gamma_0} \ldots M_{\gamma_{k-1}}$ is positive.
\begin{thm} \label{thm:generic_ue_and_separation}
Let $(G,\tau)$ be an $S$-adic graph and let $\mu$ an ergodic probability measure on $\Sigma_G$.
Assume that there exists a positive path $p$ in $(G,\tau)$ whose associated cylinder has
positive mass for $\mu$.
Then, for $\mu$-almost every sequence $\gamma \in \Sigma_G$, the corresponding $S$-adic system is uniquely ergodic.

Furthermore, if all substitutions $\tau_e$ on the edges of $G$ are invertible and the matrix cocycle
associated with the incidence matrices of the substitutions is $\log$-integrable, then
$\theta^\mu_1 > 0$ and $\theta^\mu_1 > \theta^\mu_2$.
\end{thm}

\begin{proof}
According to Theorem~\ref{thm:invariant_measures}, uniform letter frequencies for the $S$-adic system associated with $\gamma \in \Sigma_G$ is equivalent to the fact that the sequence of nested cones $A_n (\gamma) \RR^d_+$ tends to a half-line $\RR f(\gamma)$.
Let $B = M_{\gamma_0} \ldots M_{\gamma_{k-1}}$ be the positive matrix associated with the path $p = \gamma_0 \ldots \gamma_{k-1}$.
We can use the Birkhoff ergodic Theorem to see that $\mu$-almost every path passes infinitely often through the path $\gamma$. We can hence apply Theorem~\ref{thm:furs}, and get that $A_n(\gamma) \RR_+^d$ contracts almost everywhere to a cone. Moreover, because of  the non-negativity of the matrices, we get that the exponential growth of $\log \|A_n(\gamma)\|$ is at least of the order of $\log \|B^{\lfloor n \mu([p]) \rfloor} \|$. It follows that $\theta^\mu_1 > 0$.

To prove that $\theta^\mu_1 > \theta^\mu_2$, we only need to remark that $B$ induces a contraction of the Hilbert metric (by positivity). Therefore, the sequence of nested cones $A_n(\gamma) \RR^d_+$ shrinks exponentially fast toward the generalized eigendirection. As it can be seen with the definition~\eqref{eq:theta2}, this exponential rate is $\theta_2^{\mu} - \theta_1^{\mu} < 0$.
\end{proof}

\begin{rem}\label{rem:Lya}
  A natural operation of contraction may be applied to the $S$-adic system $(G,\tau)$. We choose a family of finite paths $(\gamma^{(i)})_{i \in I}$ on $G$ such that any infinite sequence in $\Sigma_G$ may be uniquely written as a concatenation of the paths $\gamma^{(i)}$. We get a new subshift on a graph whose edges are naturally identified with $I$. On an individual $S$-adic word, such an operation corresponds to the composition of some family of substitutions. In terms of Bratteli diagrams, it is a so-called \emph{contraction}. For continued fraction algorithms, it is a so-called \emph{acceleration}. We notice that the ratios $\theta_k^\mu/\theta_1^\mu$ are invariant under contraction of $S$-adic graphs (with the appropriate measures).

An example of such acceleration is the Gauss continued fractions compared to the Farey map (see also~\cite{Schweiger00} for multidimensional continued fraction algorithms). Contraction is sometimes needed, as in the case of standard continued fractions: the unique invariant measure absolutely continuous with respect to Lebesgue for the Farey map has infinite mass. It is hence impossible to define Lyapunov exponents. For its acceleration, that is, the Gauss map, the measure is finite. It can be checked that the matrix cocycle, despite it is unbounded, is $\log$-integrable and Lyapunov exponents may be defined over the Gauss map. 

The same phenomenon as for standard continuous fraction is true for interval exchanges: the unique invariant measure absolutely continuous with respect to the Lebesgue measure for the Rauzy induction is infinite~\cite{Veech}. Nevertheless, under some proper acceleration, that is, the Zorich acceleration, the measure becomes finite and the cocycle is integrable~\cite{Zorich96}.
\end{rem}

\subsection{Simultaneous approximation and cone convergence} \label{subsec:cf}
If we follow the vocabulary of Markov chains~\cite{Seneta}, or of continued fractions~\cite{BRENTJES,Schweiger00}, it is natural to consider the following definitions.
\begin{dfn}[Weak and strong convergence]
Let $f$ be the generalized eigenvector for an $S$-adic system on the alphabet $\AAA = \{1,\ldots,d\}$, normalized by $f_1 + \ldots + f_d = 1$. Let $(e_1, \ldots, e_d)$ be the canonical basis of $\RR^d$. Let $(M_n)_n$ stand for the sequence of incidence matrices associated with its directive sequence, and note $A_n = M_0 \cdots M_{n-1}$. 
The $S$-adic system $X$ is \emph{weakly convergent} toward the non-negative half-line directed by $f$ if
$$
\forall i \in \{1,\ldots,d\}, \qquad \lim_{n \to \infty} \dist \left(\frac{A_n e_i}{\|A_n e_i\|_1}, f \right) = 0.
$$
It is said to be \emph{strongly convergent} if for a.e. $f$
$$
\forall i \in \{1,\ldots,d\}, \qquad \lim_{n \to \infty} \dist(A_n e_i, \RR f) = 0.
$$
\end{dfn}
If, for any $i$, $\|A_n {e_i}\|$ tends to infinity as $n$ tends to infinity, then the strong convergence for the vector $f$ is equivalent to the fact that the sequence of nested cones $A_n \RR^d_+$ tends to the non-negative half-line generated by $f$: the points $A_n e_i / \|A_n e_i\|$ are exactly the extremal points of the intersection of the cone $A_n \RR^d_+$ with the set of vectors of norm $1$. Moreover, if $\theta$ is the angle between $A_n e_i$ and $f$, then
\[
\dist \left( \frac{A_n e_i}{\|A_n e_i\|_1}, f \right) = 2 \sin\left(\frac{\theta}{2} \right)
\quad \text{and} \quad
\frac{\dist (A_n e_i, \RR f)}{\|A_n e_i\|_2} = \sin(\theta).
\]
In particular, the asymptotics of both quantities agree at the first order. From this remark it is easy to show that Lyapunov exponents may be used to give a general estimate on the speed of convergence.
\begin{thm} \label{thm:cone_convergence_speed}
 Let $(G,\tau)$ be an $S$-adic graph whose substitutions have   invertible incidence matrices.
 Let $\mu$ be an ergodic probability measure on $\Sigma_G$ for which the cocycle $A_n$ associated
 with $(G,\tau)$ is $\log$-integrable and assume that there exists a positive path in $(G,\tau)$
 whose associated cylinder has positive mass for $\mu$.
 
 Then, for $\mu$-almost every $S$-adic sequence in $\Sigma_G$, the associated $S$-adic system is weakly convergent. More precisely,  if $f(\gamma)$ denotes the generalized eigenvector of a $\mu$-generic sequence $\gamma$, then 
 \[
 \forall i \in \{1,\ldots,d\}, \qquad \lim_{n \to \infty} \frac{1}{n} \log \dist \left(\frac{A_n e_i}{\|A_n e_i\|_1}, f (\gamma)\right) = \theta^\mu_2 - \theta^\mu_1.
 \]
 
 Moreover, if $\theta_2^\mu < 0$, then, for $\mu$-almost every $S$-adic sequence in $\Sigma_G$, the associated $S$-adic system is strongly convergent. More precisely, 
 \[
 \lim_{n \to \infty}\, \max_{i \in \AAA} \frac{\log \dist(A_n(\gamma) e_i, \RR f(\gamma))}{ n} = \theta_2^\mu.
 \]
\end{thm}
For the proof, see~\cite{Lagarias} where a similar result is shown in the context of Diophantine approximation.
In particular, the quantity $1 - \frac{\theta_2^\mu}{\theta_1^\mu} = \frac{1}{\theta^\mu_1} (\theta^\mu_1 - \theta^\mu_2)$
is expressed as the uniform approximation exponent for unimodular continued fractions algorithms such as the Jacobi-Perron
algorithm (the coefficient $1/\theta^\mu_1$ which does not appear in Theorem~\ref{thm:cone_convergence_speed} is here to
take care of the size of the denominators); see also \cite{Baladi,Baldwin} in the same vein.

\subsection{Balancedness, discrepancy and deviations of Birkhoff averages} \label{subsec:deviation_Birkhoff_and_balancedness}
Recall that if $\phi: \AAA \rightarrow \RR$ is a function, we  define the Birkhoff sums of $\phi$ along an infinite word $u \in \AAA^\NN$ as
\[
S_n(\phi,u) = S_n (\phi, T,u)= \phi(u_0) + \ldots + \phi(u_{n-1}).
\]
We consider mostly words whose associated shifts are uniquely ergodic. In that case, the Birkhoff averages $S_n(\phi,u) / n$ converge to the mean of $\phi$ with respect to the frequencies $\sum_{i \in \AAA} \phi(i) f_i$. To get a convergence speed for $S$-adic dynamical systems, a scheme was developed by A. Zorich~\cite{Zorich97} in the context of interval exchanges. It was then further generalized in~\cite{Forni02} and~\cite{Bufetov} (see also~\cite{DelecroixHubertLelievre}). For substitutive systems, a more precise result is developed by B. Adamczewski~\cite{Adam4} and also  more recently in~\cite{BBH}. A particular case of all these results says that the convergence speed of Birkhoff averages is governed by the ratio of the two first Lyapunov exponents.
\begin{thm} \label{thm:deviation_Birkhoff_averages}
 Let $(G,\sigma)$ be an $S$-adic graph on the alphabet $\AAA$ whose substitutions have invertible incidence matrices.
 Let $\mu$ be an ergodic probability measure on $\Sigma_G$ for which the cocycle $A_n$ is $\log$-integrable.
 Assume that there is a positive path in $G$ whose  associated cylinder has positive mass for $\mu$.
 Let $\theta^\mu_1$ and $\theta^\mu_2$ denote the two first Lyapunov exponents of $A_n$ with respect to $\mu$.

 If $\theta_2^\mu \geq 0$, then, for $\mu$-almost all $\gamma$ in $\Sigma_G$ we have, uniformly in $u \in X_G(\gamma)$,
 \[
 \limsup_{n \to \infty} \ \max_{i \in \AAA} \frac{\log \left||u_0 \ldots u_{n-1}|_i - n f_i\right|}{\log n} = \frac{\theta^\mu_2}{\theta^\mu_1},
 \]
 where $(f_i)_{i \in \AAA}$ is the letter frequency vector  of $X_G(\gamma)$.

 If $\theta_2^\mu < 0$, then, for $\mu$-almost all $\gamma$ in $\Sigma_G$, there exists a constant $C = C(\gamma)$ such that if $f = (f_i)_{i \in \AAA}$ denotes the frequency vector of $X_G(\gamma)$, then for every letter $i \in \AAA$, every word $u$ in $X_G(\gamma)$ and every $n$,  we have
 \[
 \left||u_0 \ldots u_{n-1}|_i - n f_i \right| \leq C,
 \]
 where $(f_i)_{i \in \AAA}$ is the letter frequency vector  of $X_G(\gamma)$. In particular, each word in $X_G(\gamma)$ is $C$-balanced.
\end{thm}
Note that a characterization of balanced words generated by primitive substitutions is given in~\cite[Corollary 15]{Adam0}. It is shown that there exist balanced fixed points of substitutions for which $\theta_2=0$, i.e.,  the incidence matrix of the substitution has an eigenvalue of modulus one. It is even proved that if   a primitive substitution generates  an infinite word  that  is balanced, then all eigenvalues of modulus one of the incidence matrix have to be roots of unity. 
Furthermore, observe that the Thue-Morse word
is  $2$-balanced, but if one considers generalized balances with respect to   factors of length $2$ instead of letters, then it is not balanced anymore: according to the notation of  \cite{Adam0},  one has,  for the  eigenvalue $-1$ of $\sigma_2$,  $A_{\sigma_2,U_2}=1$  (and not   $A_{\sigma_2,U_2}=0$ as  claimed in \cite{Adam0}), which leads  to  (unbounded) imbalances for the factor $aa$.

\begin{rem}
Let us remark that the first part of the theorem might fail for some letter in the alphabet. Namely consider the substitution $\sigma$ on $\{a,b,c\}$ given by $a \mapsto abc,\, b \mapsto bba,\, c \mapsto cca$. Then the eigenvalues of its incidence matrix are $3$, $2$ and $0$. It is easy to show that $b$ and $c$ have discrepancy of order  $n^{\log(2)/\log(3)}$, but the letter $a$ has bounded discrepancy.
In general, for $S$-adic systems, in order to provide a precise asymptotic for each letter,  one needs to care about the  other Lyapunov exponents and the associated Lyapunov flags.
\end{rem}

\begin{proof}
We only sketch the proof which is very similar to what is done in~\cite{Zorich97}, \cite{DelecroixHubertLelievre}, \cite{Adam4} or \cite{delecroixwords}.
According to Theorem~\ref{thm:generic_ue_and_separation}, the existence of the positive path with positive mass implies that $\theta_1^{\mu} > \theta_2^{\mu}$ and that $\mu$-almost every $S$-adic path in $\Sigma_G$ gives a uniquely ergodic $S$-adic system.
In particular, letters have uniform frequencies. Let $(\tau_{\gamma_n})_n$ be a directive sequence of substitutions which satisfies the Oseledets Theorem, $X_{\gamma}$ the associated shift space, and $f$ the associated letter frequency  vector.
Let $W_n = \{\tau_{[0,n)}(a) \mid a \in \AAA \}$ be the set of images of letters by $\tau_{[0,n)} = \tau_0 \ldots \tau_{n-1}$.
Then, by definition of Lyapunov exponents, for every $\epsilon > 0$, for  $n$ large enough, for all $w \in W_n$ and all $i \in \AAA$, we have
\[
\exp(n (\theta_1^\mu-\epsilon) \leq |w| \leq \exp(n (\theta_1^\mu+\epsilon))
\quad \text{and} \quad
| |w|_i - |w| f_i| \leq \exp(n (\theta_2^\mu+\epsilon)).
\]
In particular, we get that for $w \in W_n$ with $n$ large enough
\[
\frac{\log (|w|_i- |w| f_i|)}{\log(|w|)} \leq \frac {\theta_2^\mu+\epsilon}{\theta_1^\mu-\epsilon}.
\]
This proves that the theorem holds for the elements of $W_n$. Now any word in $X$ may be decomposed with respect to the building blocks $W_n$ (according to the Dumont-Thomas prefix-suffix decomposition \cite{DT89}). From that decomposition, it remains to perform a summation to obtain the theorem.

The second part follows from Theorem~\ref{thm:ness_criterion}.
\end{proof}

Recall that, a substitution is said \emph{irreducible Pisot} if the characteristic polynomial of its incidence matrix is the minimal polynomial of a Pisot number, that is, a real algebraic integer larger than $1$ whose other Galois conjugates are smaller than $1$ in modulus. We thus say that $(G, \sigma, \mu)$ satisfies the \emph{Pisot condition} if 
\[
\theta_1 ^{\mu}> 0 > \theta_2 ^{\mu}.
\]
By Theorem~\ref{thm:deviation_Birkhoff_averages}, an $S$-adic graph  (satisfying the general   assumptions of the theorem) endowed with a measure $\mu$ such that the Lyapunov exponents satisfy the Pisot condition is such that for $\mu$-almost every $\gamma$, the associated $S$-adic system $X_G(\gamma)$ is made of balanced words.
This property is known to hold for some continued fraction algorithms endowed with their absolute continuous measure: the standard continued fractions, the Brun algorithm in dimension 3~\cite{Brun}, the Jacobi-Perron algorithm in dimension 3~\cite{BroiseGuivarch}. Some more precise results that hold for \emph{all} measures are proven in~\cite{AvilaDelecroix}.

One reason for the importance of the Pisot condition is due to  the Pisot conjecture. Recall that irreducible Pisot substitutions are conjectured to have pure discrete spectrum.  Pisot substitutive dynamical systems are thus expected to be measure-theoretically isomorphic to a translation on a compact Abelian group. In~\cite{BST} Rauzy fractals are associated with Pisot $S$-adic systems, and pure discrete spectrum is proved under a generalized geometric coincidence condition. This applies to Brun and Arnoux-Rauzy continued fractions: almost every of these expansions has pure discrete spectrum.

\section*{Acknowledgements}
We would like to thank warmly  S\'ebastien Labb\'e, Julien Leroy, Thierry Monteil, N.~Pytheas Fogg  and  Wolfgang Steiner  for their careful reading and stimulating discussions.

\end{document}